\documentclass[11pt]{article}
\usepackage{amsmath}
 \usepackage{amscd}
 \usepackage{amssymb}
 \usepackage{theorem}

\usepackage{color}

 \usepackage{graphicx}
 \usepackage{mathptmx}       
\usepackage{amsmath}
 \usepackage{amscd}
\usepackage{amssymb}
  \usepackage[all]{xy}

\newtheorem{TEO}{Theorem}[section]
\newtheorem{PROP}[TEO]{Proposition}
\newtheorem{LEM}[TEO]{Lemma}
\newtheorem{DEF}[TEO]{Definition}
\newtheorem{COR}[TEO]{Corollary}

\newtheorem{teoa}{Theorem A}

\newtheorem{teob}{Theorem B}

\newtheorem{teorc}{Theorem C}

\theorembodyfont{\normalfont}
\newtheorem{EX}[TEO]{Example}

\newtheorem{REM}[TEO]{Remark}
\theorembodyfont{\normalfont}

\newcommand\Oh{{\mathcal O}}

\newcommand\sF{{\mathcal F}}
\newcommand\sG{{\mathcal G}}

\newcommand\sI{{\mathcal I}}
\newcommand\sJ{{\mathcal J}}

\newcommand\sM{{\mathcal M}}
\newcommand\sN{{\mathcal N}}

\newcommand\om{\omega}

\newcommand\Ga{\Gamma}

\newcommand\fie{\varphi}

\newcommand\simby[1]{\buildrel{{\scriptstyle\mathrm{#1}}}\over{\sim}}
\newcommand\numeq{\simby{{num}}}         
\newcommand\lineq{\simby{{lin}}}         

\newcommand\dual{\mathrel{\raise3pt\hbox{$\underline{\mathrm{\thinspace d
\thinspace}}$}}}

\newcommand\iso{\cong}

\newcommand\into{\hookrightarrow}
\newcommand\onto{\twoheadrightarrow}

\newcommand\C{\mathbb C}
\newcommand\proj{\mathbb P}
\newcommand\Ka{\mathbb K}
\newcommand\Na{\mathbb N}

\newcommand\red{{\operatorname{red}}}

\newcommand\length{\operatorname{length}}

\newcommand\im{\operatorname{Im}}
\newcommand\rank{\operatorname{rank}}

\newcommand\Ann{\operatorname{Ann}}

\renewcommand\div{\operatorname{div}}

\newcommand\Hom{\operatorname{Hom}}

\newcommand\sHom{\operatorname{{\mathcal H}{\it om}}}

\newcommand\Sing{\operatorname{Sing}}

\newcommand\Supp{\operatorname{supp}}

\newcommand\cliff{\operatorname{Cliff}}
\newenvironment{proof}[1][]{\noindent\textbf{Proof#1}.  }{{\hfill $\blacksquare$}}

\begin{document}

\title{On Clifford's theorem for singular curves
\thanks{This research was partially supported  by Italian MIUR through PRIN 2008 project  ``Geometria delle variet\`a algebriche e dei loro spazi di moduli". }}
\author{Marco Franciosi, Elisa Tenni}
\date{}

\maketitle

\begin{abstract}
Let $C$ be a  2-connected  projective  curve  either reduced with planar singularities  or contained in a smooth  algebraic surface and let $S$  be a  subcanonical cluster (i.e. a 0-dimensional scheme such that the space $H^0(C, \sI_S K_C)$ contains a generically invertible section). Under some general assumptions on $S$ or $C$ we show  that  $h^0(C,\sI_S K_C) \leq p_a(C) - \frac{1}{2} \deg (S)$ and if  equality holds then either $S$ is trivial,  or  $C$ is honestly hyperelliptic  or  3-disconnected.

As a corollary we give a generalization of Clifford's theorem for  reduced curves with planar singularities.

\hfill\break	
{\bf keyword:} algebraic curve,  Clifford's theorem, subcanonical cluster

\hfill\break  {\bf Mathematics Subject Classification (2010)} 14H20,  14C20, 14H51
\end{abstract}

\section{Introduction}

Since the early days of algebraic geometry the rule of residual series turned out to be  fundamental
 in studying the geometry of a projective variety. The first  results of the German school (Riemann, Roch, Brill, Noether, Klein, etc...)   on special divisors
were indeed based on the deep analysis of a  linear series  $|D|$  and its  residual $|K-D|$.

The purpose of this paper is to extend this basic approach to the analysis of special linear series defined on an algebraic curve (possibly singular, nonreduced or reducible), giving applications to the case of semistable curves.

In this paper,
in particular  we generalize the Theorem of Clifford,  which states that $$\dim |D| \leq \frac{\deg D}{2}$$
  for  every  special  effective  divisor  $D$ on a  smooth curve $C$ (see  \cite{Cl}).

One can find in the literature many approaches which generalize  Clifford's theorem and other   classical results to
 certain kinds of  sin\-gular curves, especially nodal ones. Important results  were given by D. Eisenbud and J. Harris (see \cite{EH} and the appendix in \cite{EKS}) and more recently by E. Este\-ves (see \cite{EM02}), applying essentially degeneration techniques, in the case of reduced curves with two components. See also the case of graph curves by D. Bayer and D. Eisenbud in \cite{BE}.
 L. Caporaso  in
 \cite{CAP} gave a generalization of Clifford's theorem for certain line bundles on stable curves, in particular line bundles of degree at most 4 and line bun\-dles whose degree is bounded by $2 p_a(\Gamma_i)$ for every component $\Gamma_i$.

Our approach is more general  since we
 deal with rank one torsion free sheaves on possibly reducible and non reduced curves, without any bound  on the number of components, but with very natural assumptions on the multidegree of the sheaves we consider.

 Our analysis focuses on 2-connected curves,
 keeping  in mind
the classical characterization of  special
divisors on algebraic curves as effective divisors contained in the canonical system.  To this purpose  we introduce
 the notion of {\em  subcanonical cluster}, i.e.  a 0-dimensional subscheme $S  \subset C$  such that  the space $H^0(C, \sI_S K_C)$ contains a generically invertible section
 (see Section \ref{subproperties} for definition and main properties).

We recall that a curve $C$ is {\em m-connnected} if $\deg_B K_C \geq m + (2p_a(B) - 2)$ for every subcurve $B\subset C$, or equivalently $B\cdot (C-B)\geq m$ if $C$ is contained in a smooth surface.

From our point of view it is fundamental
 to work only with subcanonical clusters since our  aim is to consider only clusters truly contained in a canonical divisor. Moreover we need to avoid  clusters contained in a hyperplane canonical section but with uncontrolled behavior. For instance by automatic adjunction (see \cite[Lemma 2.4]{CFHR}) a section vanishing on a component $A$ such  that $C=A+B$ yields a
section in $H^0(B,K_B)$, but considering the embedding $H^0(B, K_B)\into H^0(C, K_C)$, we can build
clusters with unbounded degree on $A$  such that every section in $H^0(C, K_C)$ vanishing  on them vanishes on the entire subcurve $A$.

Our main result is the following theorem.

\begin{teoa}\label{cliffordglobale}
Let $C$ be a projective curve  either   reduced with planar singularities  or contained in a smooth  algebraic surface. Assume $C$ to be
 2-connected
 and let $S \subset C$ be a  subcanonical cluster. Assume one of the following holds:
 \begin{enumerate}
  \renewcommand\labelenumi{(\alph{enumi})}
\item $S$ is a Cartier divisor;
\item there exists $H \in H^0(C, \sI_S K_C)$ such that $\div(H) \cap \Sing(C_{\red})=\emptyset$;
\item $C_{\red}$ is 4-connected.
\end{enumerate}

Then
 $$h^0(C,\sI_S K_C) \leq p_a(C) - \frac{1}{2} \deg (S)$$
 Moreover if equality holds then the pair $(S, C)$ satisfies one of the following assumptions:

 \begin{enumerate}
  \renewcommand\labelenumi{(\roman{enumi})}
\item $S= 0,\, {K_C}$;
\item  $C$ is honestly hyperelliptic and $S$ is a multiple of the honest $g_{2}^{1}$;
\item $C$ is 3-disconnected (i.e.  there is
a decomposition $C=A+B$ with  \mbox{$A\cdot B=2$)}. \\
\end{enumerate}
\end{teoa}

 Let $\cliff(\sI_S K_C):= 2p_a(C) -  \deg(S) - 2\cdot  h^0(\sI_S K_C)$ be the Clifford index of the
sheaf  $\sI_S K_C$.
Notice that if $S$ is a Cartier divisor then  $\cliff(\sI_S K_C)$ is precisely the classical Clifford index for invertible sheaves.
Theorem A is equivalent to the statement that the Clifford index is non negative.

If $C$ is a smooth curve the theorem is equivalent to the classical {\em Clifford's theorem}, while if $C$ is 1-connected but 2-disconnected then $|K_C|$ has base points and
 therefore the cluster consisting of such base points does not satisfy the theorem. Moreover without our assumptions the theorem is false even for subcanonical clusters contained in  curves with very ample canonical sheaf. See for instance Example \ref{tetraedro}. However we obtain a more general inequality by adding a correction term bounded by half of the number of irreducible components of $C$. See Theorem \ref{numero_irriducibili} for the full result.\\

The proof is based on the analysis of a cluster $S$ of minimal Clifford index and maximal degree and of its residual $S^{\ast}$ (see Subsection 2.3 for  definitions and main properties). When considering the restriction to $C_{\red}$ it may happen that every section in $H^0(C, \sI_S K_C)$ decomposes as a sum of sections with small support. This behaviour is completely new with respect to the smooth case and can even lead to the existence of clusters with negative Clifford index. This is the reason why in Section 2.3 we introduce the notion of splitting index of a cluster and we run our analysis by a stratification of the set of subcanonical clusters by their splitting index.

For clusters in each strata with minimal Clifford index the following dichotomy holds: either $S^{\ast} \subset S$ or $S$ and $S^{\ast}$ are Cartier and disjoint. In the first case we estimate the rank of the restriction of $H^0(C, \sI_S K_C)$ to the curve supporting $S$, while in the second case we give a generalization of the classical techniques  developed by Saint Donat in  \cite{sd}.\\

 As a corollary of Theorem A  we are able to  more deeply analyze the case of reduced curves since the intersection products are always nonnegative.
The following results apply in particular to the case of 4-connected semistable curves.

 \begin{teob} Let $C$ be a projective $4$-connected reduced  curve with planar singularities. Let $L$ be an invertible  sheaf and $S$ a cluster on $C$. Assume that

$$0 \leq \deg[(\sI_S L)_{|B}] \leq \deg {K_C}_{|B}$$
for every subcurve $B \subset C$. Then
$$h^0(C,\sI_S L )\leq \frac{\deg{\sI_S L}}{2}+1.$$
Moreover if equality holds then $\sI_S L \cong \sI_T \omega_C$ where $T$ is a subcanonical cluster. The pair $(T, C)$ satisfies one of the following assumptions::
 \begin{enumerate}
  \renewcommand\labelenumi{(\roman{enumi})}
\item $T=0,\, K_C$
\item  $C$ is honestly hyperelliptic and $T$ is a multiple of the honest $g_{2}^{1}$.
\end{enumerate}
\end{teob}
In the case of smooth curves an effective divisor $D$ either satisfies the assumptions of Clifford's theorem, or it is non special and $h^0(C, D)$ is computed easily by means of Riemann-Roch Theorem. If the curve $C$ has many components we may have a mixed behavior, which we deal with in the following theorem.
 \begin{teorc} Let $C$ be a projective $4$-connected reduced  curve with planar singularities. Let $L$ be an invertible  sheaf and $S$ a cluster on $C$ such that
$$0 \leq \deg[(\sI_S L)_{|B}]
\mbox{ \ for every subcurve   } B \subset C .$$
Assume there exists a subcurve $\Gamma\subset C$ such that $ \deg ({K_C}_{|\Gamma}) < \deg ( \sI_S L_{|\Gamma})$.
and
let $C_0$ be the maximal subcurve such that $$\deg[(\sI_S L)_{|B}] > \deg {K_C}_{|B}
\mbox{ \ for every subcurve   } B \subset  C_0 .$$
Then
$$h^0(C,\sI_S L )\leq \frac{\deg{\sI_S L}}{2} + \frac{\deg(\sI_S L - K_C)_{|C_0}}{2}.$$

 \end{teorc}

 \hfill\break
 We believe that the above results may be  useful for  the study of vector bundles on the compactification of the Moduli Space of   genus $g$ curves
 and in particular to  the analysis  of limit series.
Moreover  they may be considered as a first step in order to develop a  Brill-Noether  type analysis for  semistable curves.
Further  applications will be given  in a forthcoming article (see \cite{FrTe2}) in which  we  analyze the  normal generation of
invertible sheaves on numerically connected curve.  In particular we are going to give a  generalization of Noether's Theorem.

Finally,
 as shown in \cite{CFHR},   the study of invertible sheaves on curves  lying on a smooth algebraic surface
is rich in implications
when Bertini's  theorem  does not hold or  simply
if one needs to consider every curve contained  in a given linear system.

\hfill\break
The paper is organized as follows. In Section 2 we set the notation and prove some preliminary results. In Section 3 we prove Theorem A, in Section 4 we study the case of reduced curves and prove Theorem B and C. Finally in Section 5 we illustrate some examples in which we illustrate that the Clifford index may be negative if our assumptions are not satisfied.

\hfill\break
{\bf Acknowledgments.}
The authors would like to thank  Pietro Pirola for his stimulating suggestions. We wish to thank the Referee for the careful reading and for the very useful
comments and corrections. The second  author  wishes to thank  the Department of Mathematics
 of the  University of Pisa  for providing an excellent research environment.

\section{Notation and Preliminary results}

\subsection{Notation and conventions}

 We work over  an algebraically closed field $\Ka$ of characteristic $\geq 0$.

Throughout this paper a curve $C$ will always be a Cohen-Macaulay scheme of pure dimension 1. Moreover, if not otherwise stated, a curve $C$ will be \textit{projective,  either reduced  with planar singularities}
 (i.e. such that for every point $P\in C$ it is $ \dim_{\Ka} \sM/\sM^2 \leq 2$ where $\sM$ is the maximal ideal of $\Oh_{C,P}$)
\textit{or contained in a smooth algebraic surface}  $X$, in which case we allow $C$ to be reducible and non reduced.

In both cases we will use the standard notation for curves lying on smooth algebraic surface, writing $C=\sum_{i=1}^{s} n_{i} \Gamma_{i}$, where
 $\Gamma_{i}$ are    the
 irreducible components of $C$  and $n_{i}$ are their multiplicities.

\textit{A subcurve $B\subseteq C$ is a Cohen-Macaulay subscheme of pure dimension 1};  it
will be written as
 $\sum m_{i} \Gamma_{i}$, with $0\leq m_i\leq n_i$ for every $i$.

\hfill\break
Given a sheaf $\mathcal{F}$ on $C$, we write $H^0(B, \mathcal{F})$ for $H^0(B, \mathcal{F}_{|B})$  and $H^0(C, \mathcal{F})_{|B}$ for the image of the restriction map $H^0(C, \mathcal{F}) \to H^0(B, \mathcal{F}_{|B})$.

 $\om_C$ denotes the   dualizing sheaf of $C$ (see \cite{Ha}, Chap.~III, \S7), and
$p_a(C)$ the arithmetic genus of $C$, $p_a(C)=1-\chi(\Oh_C)$.  $K_C$ denotes the canonical divisor.

\hfill\break
By abuse of notation if $B \subset C $ is a subcurve  of $ C$,   $ C-B$  denotes the  curve $A$ such that
 $C= A+B$.

 Notice that under our assumptions every subcurve $B \subseteq C$ is \textit{ Gorenstein},  which is equivalent to say that  $\omega_B$ is an invertible sheaf.

 Throughout the paper we will use the  following exact sequences:
 \begin{equation}0 \to \omega_A  \to \omega_C \to {\omega_C} _{|B} \to 0 \end{equation}
  \begin{equation}0 \to \Oh_A(-B) \to \Oh_C \to \Oh_B \to 0,  \end{equation}
 where  $\Oh_A(-B) \iso \Oh_A \otimes \Oh_X(-B)$ if $C$ is contained in a smooth surface  $X$ and corresponds to $\sI_{A\cap B} \cdot \Oh_A$ if $C$ is reduced. See \cite[Chapter 3]{miles} and \cite[Proposition II.6.4]{BPV}.

\begin{DEF}  If $A,\, B$ are subcurves of $C$ such that $A+B=C$, then
$$A \cdot B =
 \deg_B(K_C)- (2p_a(B)-2) =  \deg_A(K_C)- (2p_a(A)-2). $$

 If $C$ is contained in a smooth algebraic surface  $X$ this corresponds to the intersection product of curves as divisors on $X$.
  \end{DEF}
We have the key formula (cf. \cite[Exercise V.1.3]{Ha})
 \begin{equation}\label{genere A+B}  p_a(C)= p_a(A)+p_a(B) + A\cdot B -1. \end{equation}
   Following the original definition of   Franchetta
a  curve $C$ is ({\em numerically}) {\em
$m$-con\-nected} if
$C_1 \cdot C_2 \geq m $
for every decomposition  $ C=C_1 + C_2$ in effective, both nonzero curves.
To avoid ambiguity between the various notions of connectedness for a curve, we will say that a curve is {\em numerically
connected} if it is 1-connected,  and {\em topologically
connected} if it is connected as a topological space (with the Zariski topology).

\hfill\break
Let $\sF$ be a rank one torsion free sheaf on $C$.  We write
$\deg \sF_{|C}$ for the degree of $\sF$ on $C$,
 $ \deg \sF_{|C}=\chi(\sF)-\chi(\Oh_C) $. By Serre duality we mean
Grothendieck-Serre-Riemann-Roch  duality theorem:
  \begin{equation}
H^1(C,\sF)\dual\Hom(\sF,\om_C)
 \nonumber
 \end{equation}
(where $\dual$ denotes isomorphic to the dual space).

If $C=\sum n_i\Gamma_i$ then
 for each $i$ the natural inclusion map $\epsilon_i:\Ga_i \rightarrow C$
induces a map
$\epsilon_i^{\ast}:\sF \rightarrow \sF_{|\Ga_i}$. We denote by
 $d_i= \deg (\sF_{|\Ga_i}) = \deg_{\Gamma_i} \sF$ the degree of $\sF$ on each  irreducible component,
 and by  ${\bf d}:=(d_1, ..., d_s)$  the
 {\em multidegree  of $\sF$ on $C$}.  If  $B $ is a subcurve of $C$, by $\mathbf
d_{B}$  we mean the multidegree of $\sF_{|B}$.
We remark that there exists a natural partial ordering given by the multidegree.

 $\sF$ is  NEF if $d_i \geq 0$ for every $i$.

We say that two rank one torsion free sheaves $\sF$ and $\sG$ are numerically equivalent if their degrees coincide on every subcurve and we will use the notation $\sF \numeq \sG$.

If $S$ and $S_1$ are linearly equivalent Cartier divisor, we will use the notation $S \lineq S_1$.

 \hfill\break
A  curve
 $C$ is {\em honestly
hyperelliptic}  if there exists a finite
morphism \mbox{$\psi\colon C\to\proj^1$} of degree $2$. In this case
 $C$ is either irreducible, or of the form $C=\Gamma_{1}+
\Gamma_2$ with $p_{a}(\Gamma_{i} ) = 0$ and
$\Gamma_1 \cdot \Gamma_2=p_{a}(C) +1$
(see    \cite[\S3]{CFHR} for a detailed treatment).
For a given point $P\in \proj^1$ $\psi^{\ast} (P)$ is a cluster of degree 2,
which we will denote by a {\em honest $g_{2}^{1}$}.\\

\begin{DEF} A {\em cluster} $Z$ of {\em degree}
$\deg Z=r$ is a \hbox{$0$-dimensional} subscheme with
$\length\Oh_Z=\dim_k\Oh_Z=r$.  The multidegree of $Z$ is defined as the opposite of the multidegree of $\sI_Z$. We consider the empty set as the degree 0 cluster.
\end{DEF}

\begin{DEF}
 The Clifford index of a rank one torsion free sheaf on $C$ is
$$ \cliff(\sF) :=  \deg(\sF)- 2h^0(C, \sF)+2 $$
If $S$ is a cluster and $\sF \iso \sI_S K_C$ then the Clifford index of $S$ may be defined as the Clifford index of  $ \sI_S K_C$ and  reads as follows:
 $$\cliff(\sI_S K_C):= 2p_a(C) -  \deg(S) - 2\cdot  h^0(\sI_S K_C)$$
\end{DEF}
If $\sF$ is an invertible sheaf (in particular if $S$ is a  Cartier divisor) then $ \cliff(\sF)$, resp. $\cliff(\sI_S K_C)$ is precisely the classical Clifford index of the line bundle $\sF$, resp. $\sI_S K_C$.\\

\subsection{ Preliminary results on projective curves}

In this section  we  recall some useful
results on invertible sheaves on projective curves.\\

In the following theorem we summarize the main applications of the
results proved in
 \cite{CFHR} on Cohen--Macaulay 1-dimensional
 projective schemes.  For a general treatment see \S 2, \S 3 of  \cite{CFHR}.

 \begin{TEO}\label{thm:curve} Let $C$ be a
Gorenstein curve, $K_{C}$ the  canonical divisor of $C$. Then
  \begin{enumerate}
	 \renewcommand\labelenumi{(\roman{enumi})}
	
 \item If $C$ is 1-connected then $H^{1}(C,K_C) \iso \Ka$.

 \item If $C$ is 2-connected and $C\not \iso \proj^{1}$ then $|K_{C}|$ is base point free.

\item If $C$ is 3-connected and $C$ is not honestly hyperelliptic (i.e.,
there does not exist a finite
morphism $\psi\colon C\to\proj^1$ of degree $2$) then $K_{C}$ is very ample.
 \end{enumerate}
 \end{TEO}
(cf.  Thm. 1.1,  Thm. 3.3,  Thm. 3.6 in  \cite{CFHR}).\\

The main instrument in the analysis of sheaves on projective curves
with several components is the
following proposition, which holds in a more general
setup.
 \begin{PROP}[ \cite{CFHR}, Lemma 2.4 ]\label{lem:adj} Let $C$ be a
 projective scheme of pure dimension 1 and  let $\sF$ be a coherent
sheaf on $C$, and $\fie\colon\sF\to\om_C$ a nonvanishing map of $\Oh_C$-modules. Set
$\sJ=\Ann\fie\subset\Oh_C$, and write $B\subset C$ for the subscheme
defined by $\sJ$. Then $B$ is Cohen--Macaulay and $\fie$ has a canonical
factorization of the form
 \begin{equation}
\sF\onto\sF_{|B}\into\om_B=\sHom_{\Oh_C}(\Oh_B,\om_C)\subset\om_C,
 \label{eq:adj}
\nonumber
 \end{equation}
where $\sF_{|B}\into\om_B$ is generically onto.
 \end{PROP}

A useful corollary of the above result is the following:
 \begin{COR}\label{h1=0} Let $C$ be a
	 pure   1-dimensional
 projective scheme, let $\sF$ be a rank 1 torsion free sheaf on $C$. Assume that
 $$ \deg (\sF)_{|B} \geq  2 p_a(B) - 1$$
 for every subcurve $B \subseteq C$.

 Then $H^1(C, \sF)=0$.
 \end{COR}

 \begin{proof} The proof is a slight generalization of the techniques used in \cite[Lemma 2.1]{CF}.

 Let assume by contradiction that $H^1(C, \sF)\neq0$.
Pick a nonvanishing section $\varphi \in \Hom(\sF, \omega_C)\iso H^1(C, \sF)^*$. By Proposition \ref{lem:adj} there exists a curve $B$ such that $\varphi$ induces an injective map $\sF_{|B} \to \omega_B$. Thus $$\deg (\sF)_{|B} \leq \deg K_B= 2 p_a(B)-2$$ which is impossible.
\end{proof}\\

During our analysis of the curve $C$ we will need to estimate the dimension of $H^0(A, \Oh_A)$ for some subcurve $A \subset C$.  To this purpose we give a slight generalization of a result of Konno and
Mendes Lopes (see \cite[Lemma 1.4]{km}).

\begin{LEM}\label{precodimspan}
Let $C$ be a projective  curve, either   reduced with planar singularities  or contained in a smooth  algebraic surface  and let $C= A+B$ a decomposition of $C$.   Assume
$A= \sum_{i=1}^h A_i$  where the $A_i$  are the topologically connected components of $A$.

  \begin{enumerate}
	 \renewcommand\labelenumi{(\roman{enumi})}
	
	 \item  if $C$ is 1-connected then $h^0(A, \Oh_A) \leq  A\cdot B $
	
	  \item  if $C$ is 2-connected then $h^0(A, \Oh_A) \leq \frac{A\cdot B}{2}$
	
	   \item if $C$ is $m$-connected with $m\geq 3$  then $h^0(A, \Oh_A)\leq   \frac{A\cdot B}{2} -  h\cdot \frac{m-2}{2} $, where $h = \#\{A_i\}$.
	   Moreover equality holds if and only if $h^0(A_i,\Oh_{A_i})=1$ and $A_i\cdot B=m$ for every component $A_i$.
	
	 \end{enumerate}

\end{LEM}

\begin{proof}  The 1-connected case  is treated in    \cite[Lemma 1.4]{km}.
We will apply the same arguments for the m-connected case  with $m \geq 2$.

If  $h^0(A, \Oh_A)=1$ the inequality holds trivially.
 If $h^0(A, \Oh_A) \geq 2$ then by  \cite[Lemma 1.2]{km} there exist a decomposition $A=A_1+A_2$
with  $\Oh_{A_1}(-A_2)$ NEF and such that the restriction map $H^0(\Oh_{A_1}(-A_2)) \to H^0(\Gamma, \Oh_{\Gamma}(-A_2))$ is injective for every irreducible $\Gamma\subset A_1$.
Since  $\Oh_{A_1}(-A_2)$ is  NEF we can conclude that
$$h^0(\Oh_{A_1}(-A_2)) \leq h^0(\Gamma, \Oh_{\Gamma}(-A_2)) \leq 1-A_2\cdot \Gamma \leq 1-A_1\cdot A_2.$$
Therefore by induction on the number of irreducible components of $A$ we get
\begin{eqnarray*}
h^0(A, \Oh_A) \leq h^0(A_2, \Oh_{A_2}) + h^0(\Oh_{A_1}(-A_2)) \leq
 \frac{A_2 \cdot(C-A_2)}{2} -  \frac{m-2}{2} +1 -A_1\cdot A_2  \\
= \frac{A\cdot (C-A)}{2} -\frac{m-2}{2}+1 - \frac{ A_1\cdot (C-A_1)}{2}  \leq
  \frac{A\cdot (C-A)}{2} -\frac{m-2}{2}+1 -\frac{m }{2}
\end{eqnarray*}
This is enough to prove  ({\em ii}).
Applying the above dimension count to every  topologically connected component of $A$ we get the inequality stated in  ({\em iii}). Moreover
if $m \geq 3$ and $h^0(A_i, \Oh_{A_i}) \geq 2 $ for a topologically connected component  $A_i \subset A$  then by the above computation we have
$h^0(A_i, \Oh_{A_{i}})<  \frac{{A_i}\cdot B}{2} - \frac{m-2}{2} $. Therefore  equality holds if and only if  for every $A_i$ we have  $h^0(A_i, \Oh_{A_{i}})=1$ and
$A_i\cdot(C-A_i)= A_i\cdot B=m$. \end{proof}

\subsection{Subcanonical clusters and Clifford index}\label{subproperties}

In this section we introduce the notion of subcanonical cluster and we analyze its main properties.
Notice that  our results works under the assumption $C$ Gorenstein.

\begin{DEF} \label{canonical cluster}
Let $C$ be a   Gorenstein curve.   A cluster $S \subset C$ is {\em subcanonical} if the space $H^0(C, \sI_S K_C)$ contains a generically invertible section, i.e. a section $s_0$ which does not vanish on any subcurve of $C$.

\end{DEF}

Notice that if $S$ is a general effective Cartier divisor such that the inequality  $\deg_B (S) \leq  \frac{1}{2} \deg_B (\omega_C) $  holds for every subcurve $B\subseteq C$ (or by duality such that its multidegree satisfies
$ \frac{1}{2} \deg_B (\omega_C)  \leq \deg_B (S) \leq  \deg_B (\omega_C) $  for every subcurve $B\subseteq C$)
 then by \cite{Fr}
$S$ is  a  subcanonical cluster.

\begin{DEF}\label{residual_definition}  Let $C$ be a Gorenstein curve,
$S \subset C$ be a  subcanonical cluster and let  $s_0 \in H^0(C, \sI_S K_C)$ be a generically invertible section.
The residual cluster $S^{\ast}$  of $S$  with respect to $s_0$ is defined by  the following exact sequence
 $$\xymatrix{
 0  \ar[r]  &  \sHom(\sI_S \omega_C, \omega_C) \ar[r]^{\alpha}   & \sHom(\Oh_C, \om_C)  \ar[r]   & \Oh_{S^{\ast}}  \ar[r] & 0
 }$$
 where the the map  $\alpha$ is defined by $\alpha(\varphi): 1 \mapsto \varphi(s_0)$.
 \end{DEF}
By  duality it is
$ \sI_{S^{\ast}} \omega_C \iso  \sHom(\sI_S \om_C, \om_C) $. Moreover,
denoting by
$\Lambda := div(s_0)$  the  effective divisor corresponding to $s_0$ we
 have  the following exact sequence
$$ 0 \to \sI_{\Lambda} \om_C \to \sI_S \omega_C \to    \Oh_{S^{\ast}}  \to 0 $$
 Therefore $S^{\ast}$ is subcanonical  since
 $s_0 \in H^0(C, \sI_{S^{\ast}} K_C)$  and it is  straightforward to see that $(S^{\ast})^{\ast} = S$.

 Notice that   if $C$ is contained in a smooth surface and
$s_0$ is transverse  to $C$ at  a point $P\in \Supp(S)$ such that $P$ is smooth for $C_{\red}$ and $C$ has multiplicity $n$ at $P$, writing  $\sI_{\Lambda} = (x) $ and   $\sI_S=(x, y^k) \subset  \Ka[x,y]/ (x,y^n)$, then $\sI_{S^{\ast}} \iso (x, y^{n-k})$.
\begin{REM}
If $S$ is a subcanonical cluster and $S^{\ast}$ is its residual  with respect to the section $s_0$, then the sheaf $\sI_{S^{\ast}} \omega_C$ is the subsheaf of $\omega_C$ given as follows:
\begin{equation}\label{iso_adj}\sI_{S^{\ast}} \omega_C = \{ \varphi(s_0) \mbox{ s. t. } \varphi \in \sHom(\sI_S \omega_C, \omega_C)\}.\end{equation}
This is clear from the analysis of the commutative diagram
$$\xymatrix{
 0  \ar[r]  &  \sHom(\sI_S \omega_C, \omega_C) \ar[r]^{\alpha}  \ar[d]^{\beta_1} & \sHom(\Oh_C, \om_C)  \ar[r] \ar[d]^{\beta_2}   & \Oh_{S^{\ast}}  \ar[r] \ar@2{-}[d]  & 0 \\
0  \ar[r]   & \sI_{S^{\ast}} \omega_C  \ar[r]     & \omega_C   \ar[r]    & \Oh_{S^{\ast}}  \ar[r]   &  0  }$$
where the the map  $\alpha$ is defined by $\alpha(\varphi): 1 \mapsto \varphi(s_0)$ and the maps $\beta_1$ and $\beta_2$ are isomorphisms.
 \end{REM}
 \begin{REM} The product map $H^0(C,  \sI_S K_C) \otimes H^0(C, \sI_{S^{\ast}} K_C) \to H^0(C, 2 K_C)$ satisfies the following commutative diagram:
 $$\xymatrix{
  H^0(C,\sI_S K_C) \otimes  \Hom(\sI_S K_C, K_C)      \ar[d]^{\beta}    \ar[r]^{ \ \ \ \ \ \ \ \ \ \ \ \ \ \ \ \ \ \text{ev}}  & H^0(C, K_C)  \ar[d]^{ \cdot s_0}  \\
  H^0(C,\sI_S K_C) \otimes H^0(C,\sI_{S^{\ast}} K_C) \ar[r]   & H^0(C,2 K_C)\\
 }$$
where the first row is the evaluation map $i \otimes \varphi \mapsto \varphi(i)$, the map $\beta$ is the isomorphism defined  by    $
\beta(i \otimes \varphi)  =i \otimes \varphi(s_0)$, and the second column is the multiplication by the section $s_0$ defining the residual $S^{\ast}$.

The diagram is commutative: on the stalks the elements $s_0 \cdot \varphi(i)$ and $i \cdot \varphi(s_0)$ must coincide.
In particular consider $ i \in H^0(C,  \sI_S K_C) $ and $ j \in H^0(C, \sI_{S^{\ast} } K_C)$:  we can write $j = \varphi (s_0)$ for some $\varphi \in \Hom(\sI_S K_C, K_C)$,  hence we have
\begin{equation}\label{prodotto} i \cdot j = i \cdot \varphi(s_0) = s_0 \cdot \varphi(i) \text{ in } H^0(C, 2K_C).
\end{equation}
\end{REM}

\begin{REM}\label{Serre}
Notice that, by Serre duality, it is  $H^1(C,\sI_S K_C) \dual H^0(C, \sI_{S^{\ast}} K_C)$, and
$\cliff(\sI_S K_C)=\cliff(\sI_{S^{\ast}} K_C)$. \\
 \end{REM}
The following technical lemmas will be useful in the proof of Theorem~A.
\begin{LEM}\label{sub-canonical}
Let $C$ be a Gorenstein curve. Let $S$, $S^{\ast}$, $T$, $T^{\ast}$ subcanonical clusters such that
 \begin{enumerate}
  \renewcommand\labelenumi{(\roman{enumi})}
\item $S^{\ast}$ is the  residual to   $S$ with respect to $H_0 \in H^0(C, \sI_{S}K_C)$
\item $T^{\ast}$ is the  residual to   $S$ with respect to $H_1 \in H^0(C, \sI_{S}K_C)$
\item T is the  residual to   $S^{\ast}$ with respect to $H_2 \in H^0(C, \sI_{S^{\ast}}K_C)$.
\end{enumerate}
Then the cluster $U$ defined as the union of $T$ and $T^{\ast}$  (i.e., $\sI_{U} = \sI_{T} \cap \sI_{T^{\ast}}$) and the cluster  defined by the intersection $R = T \cap T^{\ast}$ (i.e., $\sI_R = \sI_T + \sI_{T^{\ast}}$) are subcanonical.

\end{LEM}
\begin{proof}
$R$ is obviously subcanonical since it is contained in the   cluster $T$.

Since $H_1 \in H^0(C, \sI_{S}K_C)$ there exists an element $\varphi_1 \in \Hom(\sI_{S^{\ast}} K_C, K_C)$ such that $H_1 = \varphi_1(H_0)$ by Equation (\ref{iso_adj}). Similarly there exists $\psi_2 \in \Hom(\sI_{S} K_C, K_C)$ such that $H_2 = \psi_2(H_0)$.

By Equation (\ref{iso_adj}) $\psi_2(H_1) \in H^0(C,\sI_{T^{\ast}} K_C)$ and $\varphi_1(H_2)\in H^0(C,\sI_{T} K_C)$.

By Equation (\ref{prodotto}) we have
\begin{equation}\label{sezione_unione} H_0 \cdot  \psi_2(H_1) = H_1 \cdot \psi_2(H_0) = H_1 \cdot H_2= \varphi_1(H_0) \cdot H_2 = H_0 \cdot \varphi_1 (H_2)\end{equation}
and since $H_0, \, H_1$ and $H_2$ are generically invertible we conclude that $\psi_2(H_1) = \varphi_1(H_2)$ in $H^0(C, K_C)$ and it is generically invertible.
In particular
$$\psi_2(H_1) = \varphi_1(H_2) \in H^0(C, \sI_T K_C) \cap H^0(C, \sI_{T^{\ast}} K_C) \subset H^0(C, \sI_U K_C)$$
and we may conclude. \end{proof}

\begin{REM}
It is not difficult to prove that the clusters $T$ and $T^{\ast}$ defined in the previous Lemma are reciprocally residual with respect to the section $H_3=\psi_2(H_1) = \varphi_1(H_2) \in H^0(C, K_C)$. This induces an equivalence relation on the set of clusters with properties similar to the classical linear equivalence relation between divisors.
\end{REM}
\begin{DEF}
A nontrivial subcanonical cluster $S$ is called {\em splitting} for the linear system $|K_C|$ if for every $H \in H^0(C,\sI_S K_C)$ there exists a decomposition $H= H_1 + H_2$ with $H_1,\,H_2 \in H^0(C,\sI_S K_C)$ and a decomposition $C_{\red}=C_1 + C_2$  such that $\Supp({H_1}_{|C_\red})\subset C_1$ and $\Supp({H_2}_{|C_\red})\subset C_2$.

The {\em splitting index} of $S$ is the minimal number $k$ such that  for every element $H \in H^0(C,\sI_S K_C)$ there exists a decomposition $H= \sum_{i=0}^k H_i$ with $H_i \in H^0(C,\sI_S K_C)$ and a decomposition $C_{\red}=\sum_{i=0}^k C_i$  such that $\Supp({H_i}_{|C_\red})\subset C_i$. We define the splitting index of the zero cluster to be zero.
\end{DEF}

\begin{PROP}\label{generic_splitting}
Let $C=\sum_{i=1}^{s} n_{i} \Gamma_{i}$ be a Gorenstein curve and let $S$ be a subcanonical cluster. Then the following properties hold.

\begin{enumerate}
\item If the splitting index of $S$ is $k$ then there is a decomposition $C_{\red}=\sum_{i=0}^k C_i$ such that every $\overline{H} \in H^0(C,\sI_S K_C)$ can be decomposed as $\overline{H}= \sum_{i=0}^k \overline{H}_i$ with $\Supp({\overline{H}_i}_{|C_\red})\subset C_i$. Moreover if $\overline{H}$ is generic then the sections $\overline{H}_i$ can not be further decomposed.
\item Given the above minimal decomposition $C_{\red}=\sum_{i=0}^k C_i$ we have that $C_i \cap C_j$ is in the base locus of $|\sI_S K_C|$ for every $i$ and $j$.
\item If there exists a section $H \in H^0(C,\sI_S K_C)$ such that $\div(H) \cap (\Gamma_i \cap \Gamma_j)= \emptyset$ for every $\Gamma_i \neq \Gamma_j$ irreducible components in $C$, then the splitting index of $S$ is zero.
\end{enumerate}
\end{PROP}

\begin{proof}
To prove the first statement, since the possible decompositions of $C_{\operatorname{red}}$ are finite, there exists a decomposition $C_{\operatorname{red}}=\sum_{i=0}^{k} C_i$ such that the generic element $\overline{H} \in H^0(C,\sI_S K_C)$ decomposes as $\overline{H}= \sum_{i=0}^k \overline{H}_i$, $\Supp({\overline{H}_i}_{|C_\red})\subset C_i$. Call $Y$ the set of sections with this property, we are going to show that $Y= H^0(C,\sI_S K_C)$. $Y$ is obviously a linear subspace of  $H^0(C,\sI_S K_C)$ and, since it is dense, it must coincide with the entire space. \\

Similarly, we can prove that the subset $X$ of $H^0(C,\sI_S K_C)$ whose elements can be decomposed in at least $k+2$ summands is the union of a finite number of proper subspaces of $H^0(C,\sI_S K_C)$, hence its complement is open.\\

To prove the second statement, assume that there exists a decomposition $H= H_1 + H_2$ with $H_1,\,H_2 \in H^0(C,\sI_S K_C)$ and a decomposition $C_{\red}=C_1 + C_2$  such that $\Supp({H_1}_{|C_\red})\subset C_1$ and $\Supp({H_2}_{|C_\red})\subset C_2$. Then $H_1$ and $H_2$ vanish on $C_1 \cap C_2$, hence $H$ vanishes there too.\\

In particular if $\div(H) \cap (\Gamma_i \cap \Gamma_j)= \emptyset$ for every $\Gamma_i \neq \Gamma_j$, such a decomposition can not exist. The third statement follows easily from the second.
\end{proof}

\begin{REM}\label{splitting_aggiunto} If $S$ is a subcanonical cluster and $S^{\ast}$ is  its residual with respect to a section $H$ then their splitting indexes are the same. Indeed, $H^0(C, \sI_{S^{\ast}}K_C)= \{\varphi(H)\, s.t. \varphi \in \Hom(\sI_S K_C, K_C)\}$ and if $H$ can be decomposed as in Lemma \ref{generic_splitting}, then the same is true for $\varphi(H)$. By the symmetry of the situation we may conclude.
\end{REM}

\begin{LEM}\label{bpf}Let $C$ be a 1-connected  Gorenstein curve and  let $S$ be a non trivial subcanonical cluster with
minimal Clifford index among the clusters with splitting index smaller than or equal  to $k \in \Na$. Then $H^0(C, \sI_SK_C) $ is  base point free, that is,   for every $P\in C$
the evaluation map
$$H^0(C, \sI_SK_C) \otimes  \Oh_{C,P}  \to  {\sI_{S}}_{|P} \subset \Oh_{C,P}$$  generates   the ideal
$ {\sI_{S}}_{|P}$ as $\Oh_{C,P}-$module.
\end{LEM}
\begin{proof}
 The statement is equivalent to say that  for every subscheme  $T$  containing $S$ with $\length (T)= \length (S)+1$,
it is
  $h^0(C, \sI_TK_C) < h^0(C, \sI_SK_C)$.

  If $T$  is not subcanonical  then by definition of subcanonical cluster there exists a decomposition $C=A+B$  and a suitable cluster $T_A$ with support on $A$ such that
  $$ H^0(A, \sI_{T_A} \omega_A) \iso H^0(C, \sI_TK_C) $$
 and then we conclude since
$$ H^0(A, \sI_{T_A} \omega_A) \into  H^0(A, \sI_{S_A}\omega_A) $$
and $h^0(A, \sI_{S_A}\omega_A) < h^0(C, \sI_SK_C)$ because $S$ is subcanonical and $C$ is 1-connected.

If $T$ is subcanonical and its
splitting index is greater than $k$ then necessarily  the vector spaces $H^0(C, \sI_TK_C) $ and $H^0(C, \sI_SK_C)$  cannot be equal.

If  $T$ is subcanonical and its splitting index is smaller  than or equal  to $k$ then 
\begin{eqnarray}\nonumber
\cliff(\sI_TK_C) 
= 2p_a(C)-\deg(S)-1 - 2h^0(C, \sI_TK_C) \geq \cliff(\sI_SK_C)
\end{eqnarray}
if and only if  $h^0(C, \sI_TK_C) < h^0(C, \sI_SK_C)$.
 \end{proof}

\section{ Clifford's theorem }

In this section we will prove Theorem A.
The proof of the theorem is given arguing by contradiction by assuming the existence of a very special cluster for which its Clifford index is non-positive.

The first two lemmas works under the assumption $C$ Gorenstein. The rest of the section needs  an  assumption on  the singularities of $C$,
namely $C$ with planar singularities, or $C$ contained in a smooth algebraic surface if  non reduced.

In the following Lemma we will show that there exists a special relation between a maximal cluster with non-positive Clifford index and its residual with respect to a generic section.

\begin{LEM}\label{duality}
Let $C$ be a 2-connected   Gorenstein curve. 
Fix $k \in \Na$ and let $S$ be a nontrivial subcanonical cluster with minimal non-positive Clifford index  and  maximal total degree among the clusters with splitting index smaller than or equal to $k$. Let $S^{\ast}, \,T$, $T^{\ast}$ be  subcanonical clusters such that

 \begin{enumerate}
  \renewcommand\labelenumi{(\roman{enumi})}
\item $S^{\ast}$ is the  residual to   $S$ with respect to a generic section $H_0 \in H^0(C, \sI_{S}K_C)$
\item $T^{\ast}$ is the  residual to   $S$ with respect to a generic section $H_1 \in H^0(C, \sI_{S}K_C)$
\item $T$ is the  residual to   $S^{\ast}$ with respect to a generic section $H_2 \in H^0(C, \sI_{S^{\ast}}K_C)$.
\end{enumerate}
Then either $T^{\ast} \cap T=\emptyset$ and $\cliff(\sI_S K_C)=0$ or $T^{\ast} \subset T$.
\end{LEM}
\begin{proof}
Let $\Sigma_k$ be the set of clusters with splitting index smaller than or equal  to $k$.

Notice at first that $\deg T= \deg S$, $h^0(C, \sI_T K_C)= h^0(C, \sI_S K_C)$ and similarly for $S^{\ast}$ and $T^{\ast}$ by Remark \ref{Serre}.

$\cliff(\sI_S K_C)$ is minimal non-positive if and only if $h^0(C, \sI_S K_C) = p_a(C)- \frac{\deg S}{2} + M$ with $M \geq 0$ maximal.

Call $R$ the intersection of the two clusters $T$ and $T^{\ast}$, i.e. the subscheme defined by the ideal $\sI_{T} + \sI_{T^{\ast}}$, and $U$ the minimal cluster containing both, i.e. $\sI_U= \sI_{T} \cap \sI_{T^{\ast}}$. Then  $R$ and $U$ are subcanonical clusters by Lemma \ref{sub-canonical} and they belong
to $\Sigma_k$. Indeed by Proposition \ref{generic_splitting} and Remark \ref{splitting_aggiunto} the splitting indexes of $T$ and $T^{\ast}$ are equal to the one of $S$. Regarding $U$, by Equation (\ref{sezione_unione}) we know that there is a section $H_3 \in H^0(C, K_C)$ vanishing on $U$ such that $H_0 \cdot H_3=H_1 \cdot H_2$. Notice that, since $H_0$ and $H_1$ are generic, $H_3$ can be seen as a deformation of $H_1 \in H^0(C, \sI_{T^{\ast}} K_C)$, thus it is generic too seen as a section of $H^0(C, \sI_{T^{\ast}} K_C)$. Thus the splitting index of $U$ is smaller than or equal  to the splitting index of $T^{\ast}$. With regards to $R$, with a similar argument we can prove that $R^{\ast} \in \Sigma_k$ hence $R \in \Sigma_k$ too.

Moreover,
we have
the following exact sequence:
\[0 \rightarrow \sI_{U} \omega_C \rightarrow \sI_{T} \omega_C \oplus \sI_{T^{\ast}} \omega_C \rightarrow \sI_{R} \omega_C \rightarrow 0\]
Thus we know that \[h^0(C, \sI_{T} K_C) + h^0( C, \sI_{T^{\ast}} K_C) \leq h^0(C, \sI_{R} K_C) + h^0(C, \sI_{U} K_C).\]
By Riemann-Roch and Serre duality the L.H.S.  is equal to $ p_a(C)  + 1 + 2M $, whilst the  R.H.S. is
$\leq  p_a(C) - {\deg U}/{2} + M + p_a(C) - {\deg R}/{2} + M =  p_a(C)  + 1 + 2M $.

By the maximality of the degree of $T$   then one of the following must hold:

\begin{enumerate}\renewcommand\labelenumi{(\roman{enumi})}
\item
 $ U = K_C , \   R= 0$, whence $T \cap T^{\ast} = \emptyset$ and  $M=0$, that is  $\cliff(\sI_T K_C)=0$;
 moreover it is $h^0(C, \sI_{S} K_C) + h^0(C, \sI_{S^{\ast}} K_C)= h^0(C, K_C)+1$.

 \item $U=T,\, R=T^{\ast}$ and  in particular $T^{\ast}\subseteq T$.
 \end{enumerate} \end{proof}\\

\begin{LEM}\label{P generico} Let $C$ be a 2-connected   Gorenstein curve and $S$ be a subcanonical cluster. Assume that there is an irreducible component $\Gamma \subset C$ such that
\begin{eqnarray*}
\dim [ H^0(C, \sI_S K_C)_{|\Gamma}]  \geq 2.
\end{eqnarray*}
Then for a generic $P \in \Gamma$ the cluster $S+P$ is still subcanonical.
\end{LEM}
\begin{proof} We argue by contradiction.

If $S$ is subcanonical but $P+S$ is not, i.e. $H^0(C, \sI_{S+P}K_C)_{|B}=0$ for some subcurve $B \subset C$, (clearly $\Gamma \nsubseteq B$ since
$ H^0(C, \sI_{S+P} K_C)_{|\Gamma} \neq 0$ by our assumption), we consider the following commutative diagram

$$\xymatrix{ H^0(C-B,\sI_{P} \sI_1K_{C-B})\ar@{^{(}->}[d] \ar[r]  & H^0(C, \sI_{S+P}K_C) \ar@{^{(}->}[d] \ar[r] &  H^0(B, \sI_{S+P}K_C)_{|B}=0\\
H^0(C-B,\sI_1 K_{C-B})\ar[d] \ar[r]  & H^0(C, \sI_{S}K_C) \ar[d] \ar[r] &  H^0(B, \sI_{S}K_C)_{|B}=\mathbb{K}\\
H^0(P, \Oh_P)\ar[r]^{=} & H^0(P, \Oh_P) &
}$$
where $\sI_1$ is the ideal sheaf on $C-B$ given as the kernel of the map $\sI_S \to (\sI_S)_{|B}$.

By a simple diagram chase the restriction map $H^0(C-B,\sI_{1}K_{C-B}) \to H^0(P, \Oh_P)$ must be zero, hence by genericity of the point $P$ the global restriction map from  $H^0(C-B,\sI_{1}K_{C-B})$ to $\Gamma$ must be zero. This is impossible, since this would imply that the restriction of the global space $H^0(C, \sI_{S}K_C)$ to $\Gamma$ would be at most 1-dimensional, contradicting our assumption.
\end{proof}\\

 The following Lemma generalizes the classical techniques showed by Saint Donat in \cite{sd}.

 \begin{LEM}\label{saintdonat}
Let $C$ be a 2-connected   projective  curve, either reduced with planar singularities or contained in a smooth algebraic surface.

Fix $k \in \Na$ and let $S$ be a nontrivial subcanonical cluster with minimal non-positive Clifford index  and  maximal total degree among the clusters with splitting index smaller  than  or equal to $k$.
Let $S^{\ast}$ be the  residual to    $S$ with respect to a generic hyperplane section~$H$.

Suppose that there is an irreducible component $\Gamma \subset C$ such that
\begin{eqnarray*}
\dim [ H^0(C, \sI_S K_C)_{|\Gamma}]  \geq 2\\
\dim[  H^0(C, \sI_{S^{\ast}} K_C)_{|\Gamma} ] \geq 2
\end{eqnarray*}

Then  $S^{\ast}$ is a length 2 cluster such that $h^0(C, \sI_{S^{\ast}}K_C)= g-1$. In particular $C$ is either honestly hyperelliptic or 3-disconnected.
\end{LEM}

\begin{proof}
We divide the proof in 4 steps.

Let $\Sigma_k$ be the set of clusters with splitting index smaller than or equal to $k$. By Remark \ref{splitting_aggiunto} we know that $S^{\ast} \in \Sigma_k$.

Notice that since $C$ is 2-connected then  $2\leq \deg(S) \leq \deg(K_C) -2 $.

\hfill\break
\textbf{Step 1: $\mathbf{S}$ and $\mathbf{S^{\ast}}$ are Cartier divisor and non splitting.}

Consider a generic point $P \in \Gamma$. In particular $P \notin S$. By Lemma \ref{P generico} $P+S$ is subcanonical and by the minimality of the Clifford index  $h^0(C, \sI_P \sI_S K_C)= h^0(C,\sI_S K_C) - 1$.

Consider  a generically invertible section $H$ in $H^0(C,\sI_S K_C)$ vanishing at $P$ and the residual $S^{\ast}$ with respect to $H$. We have $P \in S^{\ast}$ and we can apply Lemma \ref{duality} because $P$ is general, hence the corresponding invertible section is general as well. Since $S^{\ast} \not\subset S$ we have  $S^{\ast} \cap S= \emptyset$ and both are Cartier divisors.

$S$ and $S^{\ast}$  Cartier with minimal Clifford indexes  among the clusters in $\Sigma_k$ implies that  both the linear systems $|K_C(-S)|$ and $|K_C(- S^{\ast})|$ are base point free by Lemma \ref{bpf}. Hence we can find a divisor $S^{\ast} \in |K_C(-S)|$ not passing through the singular locus of $C_{\operatorname{red}}$.
This implies that  the splitting index of $S^{\ast}$ is zero by Proposition \ref{generic_splitting} and Remark \ref{splitting_aggiunto} shows that the splitting index of $S$ is zero as well.

\hfill\break
\textbf{Step 2:  $\mathbf{h^0(C, \sI_S K_C)_{|D} \leq h^0(C, \sI_{S^{\ast}} K_C)_{|D}}$ for any $\mathbf{D \subset C}$.}

Consider again a generic point $P \in \Gamma$, $P \notin S$ and $P \notin S^{\ast}$. With the same argument adopted in step 1, we take a cluster $S_1^{\ast}$  residual to   $S$ such that $ P \in S_1^{\ast}$ and a secon cluster $S_2$  residual to    $S^{\ast}$ such that $P \in S_2$. By Lemma \ref{duality} $S_1^{\ast} \subset S_2$ since their intersection contains $P$. This gives us the following inequality for every subcurve $D \subset C$:
\begin{equation}
\begin{array}{rl}\label{dis_aggiunto}\dim  [ H^0(C, \sI_{S^{\ast}} K_C)_{|D}]  = & \dim [ H^0(C, \sI_{S_1^{\ast}} K_C)_{|D}]  \geq
 \dim [ H^0(C, \sI_{S_2} K_C)_{|D} ]  \\  = &  \dim [ H^0(C, \sI_{S} K_C)_{|D} ]\end{array}
\end{equation}
\hfill\break
\textbf{Step 3: $\mathbf{h^0(C, \sI_{S} K_C) = 2}$.}

We argue by contradiction, assuming that $h^0(C, \sI_{S} K_C) \geq 3$. \\
\textit{Case (a):}
$$ \exists  \mbox{ irreducible  } \Gamma \subset C  \mbox{ s. t.  }\dim [ H^0(C, \sI_S K_C)_{|\Gamma}] \geq 3.$$

We may apply Lemma \ref{P generico}  twice to conclude that, given 2 generic points $P$ and $Q$ in $\Gamma$, the cluster $P+Q+S$ is subcanonical and the points impose independent conditions  to $H^0(C, \sI_S K_C)$. Hence there exists a generically invertible $H \in H^0(C, \sI_{S} K_C)$ passing through $P + Q$.
 Consider $T^{\ast}$, the  residual to    $S$ with respect to $H$: $P+Q \subset T^{\ast}$.

Step 2 allows us to apply Lemma \ref{P generico} to the cluster $S^{\ast}$ as well, hence $P+ S^{\ast}$ is subcanonical and $P$ and $Q$ impose independent conditions to $H^0(C, \sI_{S^{\ast}} K_C)$. Hence there exists a generically invertible section $H_1 \in H^0(C, \sI_P \sI_{S^{\ast}}K_C)$ but $H_1 \notin H^0(C,\sI_Q \sI_P \sI_{S^{\ast}}K_C)$. Let $T_1$ be the  residual to   $S^{\ast}$ with respect to this section. We have that $P \in T_1$ but $Q \notin T_1$.

This is impossible: $P \in T_1 \cap T^{\ast}$ but $Q \in T^{\ast}$, $Q \notin T_1$. Thus $ \emptyset \neq T_1 \cap T^{\ast} \subsetneq T^{\ast}$ contradicting Lemma \ref{duality}.

Hence this case can not happen, that is, for every irreducible component $\Gamma$  the restriction  of $H^0(C, \sI_S K_C)$ to $\Gamma$ is at most 2-dimensional.\\
\textit{Case b:}
$$\left\{ \begin{array}{ll} \dim [ H^0(C, \sI_S K_C)_{|C_{\operatorname{red}}}] \geq 3 & \\
\dim [ H^0(C, \sI_S K_C)_{|\Gamma_0}] \leq 2 &  \text{ for every irreducible } \Gamma_0 \subset C  \end{array}\right.$$

We want to argue as in case (a)  finding  two points $P$ and $Q$ which lead to the same contradiction.

Since case (a) can not happen, we know that $\dim [ H^0(C, \sI_S K_C)_{|\Gamma}]=2$, hence there must exist a topologically connected reduced subcurve $D \supset \Gamma$, minimal up to inclusion, such that
$$\dim [ H^0(C, \sI_S K_C)_{|D}] \geq 3.$$
By minimality of $D$, there exists an irreducible component $\Gamma_1 \subset D$, with $\Gamma_1 \neq \Gamma$, and a section $H_0 \in H^0(C, \sI_S K_C)$ such that ${H_0}_{|\Gamma_1} \neq 0$ while ${H_0}_{|D-\Gamma_1}=0$. In particular $H_0$ vanishes on $\Gamma_1 \cap (D-\Gamma_1)$.

We consider a generic point $P \in \Gamma$. Thanks to Lemma \ref{P generico} and Step 1, there exists a generically invertible section $H \in H^0(C, \sI_{S+P} K_C)$ not vanishing on any singular point of $C_{\red}$.

Hence we know that the sections $H$ and $H_0$ span a 2-dimensional subspace of $H^0(C, \sI_{S+P} K_C)_{|\Gamma_1}$. We apply Lemma \ref{P generico} to $\Gamma_1$ taking a point $Q$ generic in $\Gamma_1$ such that $S+P+Q$ is subcanonical and $P$ and $Q$ impose independent conditions on $H^0(C, \sI_S K_C)$.

We may conclude as in case (a) that this case can not happen.\\

\textit{Case c:}
$$\left\{ \begin{array}{l}\dim [ H^0(C, \sI_S K_C)_{|C_{\operatorname{red}}}] =2 \\
\dim [ H^0(C, \sI_S K_C)] \geq 3 \end{array}\right.$$

Consider a generic point $P \in \Gamma$. By Lemma \ref{P generico} $S+P$ is subcanonical, and by genericity  of $P$
$$ H^0(C, \sI_{S+P} K_C)_{|C_{\operatorname{red}}}=<H>$$
where $H$ is generically invertible and does not vanish on any singular point of $C_{\red}$. In particular $P+S$ is non splitting.

We want to show that $(P + S)_{|C_{\red}}= K_{C|C_{\red}}$. If not  there would exists a point $Q$ in $C_{\red}$ not imposing any condition on $H^0(C, \sI_S K_C)$, i.e. the unique nonzero section $H \in H^0(C, \sI_{S+P} K_C)_{|C_{\operatorname{red}}}$ would vanish at $Q$. In particular $S+P+Q$ would be subcanonical, since the section $H$ must be generically invertible.
But, our assumptions are that $S$ has maximal degree among the non splitting nontrivial cluster of minimal Clifford index. Therefore, since $P+Q+S \neq K_C$ (otherwise $\dim [ H^0(C, \sI_S K_C)] \leq 2$), we should have
$$\operatorname{Cliff}(\sI_{S+P+Q} K_C) >\operatorname{Cliff}(\sI_S K_C)$$
which is equivalent to
$$h^0(C, \sI_{S+P+Q} K_C) <h^0(C, \sI_{S} K_C)-1$$
contradicting our hypotheses.

 Thus $(P + S)_{|C_{\red}}= K_{C|C_{\red}}$ and we can argue as in Step 1  taking a cluster $S_1^{\ast}$  residual to   $S$ with respect to a generic section and passing through $P$. Hence ${S_1^{\ast}}_{|C_{\red}}=P$ and   the multiplicity of $\Gamma$ in $C$ is at least 2 since $\deg S_1^{\ast} > 1 $.

In this case we consider a generic length 2 cluster $\sigma_0$ supported at $P$. Since $S$ and $S^{\ast}$ are Cartier and supported on smooth points of $C_{\red}$, it is easy to check by semicontinuity that $\sigma_0$ imposes independent conditions on $H^0(C, \sI_S K_C)$ and $H^0(C, \sI_{S^{\ast}} K_C)$, and we can treat $\sigma_0$ as we did with the length 2 cluster $P+Q$ in the previous case, that  is, we take  $T_1$ and $ T^{\ast}$ such that $P \in T_1 \cap T^{\ast}  $ but $\sigma_0 \not\subset T_1 \cap T^{\ast}$. By Lemma \ref{duality} this is a contradiction.

Hence we are allowed to conclude that
$$\dim [ H^0(C, \sI_S K_C)] =2.$$

\hfill\break
\textbf{Step 4: $\mathbf{\operatorname{\textbf{deg}} S^{\ast}=2}$ and $\mathbf{h^0(C, \sI_{S^{\ast}}K_C)=p_a(C)-1}$.}

By our assumptions and Step 3
$$0 \geq \operatorname{Cliff}(\sI_S K_C)=\deg(\sI_S K_C) -2h^0(C, \sI_S K_C)+2= \deg(\sI_S K_C)-2$$
which implies that
$$\deg S^{\ast}= \deg(\sI_S K_C) \leq 2.$$
But  if $\deg S^{\ast} = 1$   then   the point $S^{\ast}$ would be a base point for $K_C$,
which is absurd by  Theorem \ref{thm:curve} since 
 $C$ is 2-connected and has genus al least 2 since $p_a(C)=h^0(C, K_C) \geq \dim [H^0(C, \sI_S K_C)_{|\Gamma}] \geq~2$.

Finally, Riemann-Roch Theorem and Serre duality implies that
$$h^0(C, \sI_{S^{\ast}}K_C)=p_a(C)-1$$
hence $S^{\ast}$  is a length 2 cluster not imposing independent condition on $K_C$. This happens if and only if $C$ is honestly hyperelliptic or $C$ is 3-disconnected.
\end{proof}\\

The following three technical Lemmas will be used in the proof of Theorem \ref{cliffordfinale} in order to give estimates for the rank of the restriction map $r\!: H^0(C, \sI_S K_C) \to H^0(B, \sI_S K_C) $ for some particular subcurves $B \subset C$.

\begin{LEM}\label{rank1} Let $C$ be a 2-connected curve contained in a smooth algebraic surface and $S$ a non trivial subcanonical cluster with
minimal Clifford index among the clusters with splitting index smaller  than or equal  to $k \in \Na$.

If there is an irreducible component $\Gamma$  and  a point $P \in \Gamma$ such that $S_{|P}$ is not contained in $C_{\red}$, then the restriction map
$H^0(C, \sI_SK_C) \to H^0(m\Gamma, \sI_SK_C) $ has rank 1, where $m$ is the minimal integer
such that $ S_{|P} \subset m\Gamma$. 
\end{LEM}

\begin{proof} Let $S$ be a non trivial subcanonical cluster with minimal Clifford index and let $P\in C$ be a point such that $S_{|P}$ is not contained in $C_{\red}$.

Let $\Oh_{C, P} $ be the local ring of $C$ at $P$,  $\mathcal{N}$ be the maximal ideal of $\Oh_{C,P}$ and
$\sM$ be the maximal ideal of ${\Oh_C}_{{\red},P}$.

Thanks to Lemma \ref{bpf}, locally at $P$  the ideal ${\sI_{S}}_{|P} \subset  \Oh_{C, P} $  can be written as
$$ {\sI_{S}}_{|P}= ( H, H_1,\cdots,H_k, p_1,\cdots, p_l ) $$
where
$H, H_1,\cdots, H_k,p_1,\cdots ,  p_l$
 are linearly independent sections in  $H^0(C, \sI_{S} {K_C}) $.

 Moreover we ask  $H, H_1,\cdots, H_k,$ to be of
 minimal degree when restricted to  $S_{\red}$ whereas $p_1,\cdots, p_l$ must have degree strictly bigger. Algebraically,  if
 ${\sI_{S_{\red}}}_{|P} \subset \sM^n$  but $ {\sI_{S_{\red}}}_{|P} \not\subset \sM^{n+1}$, then we ask
 $H, H_1,\cdots, H_k$ to be   a basis of the $\Ka$-vector space $\frac{\sI_{S_{\red}}}{\sI_{S_{\red}}\cap\sM^{n+1}}$
and   $p_1,\cdots,  p_l$    to satisfy $ {p_i}_{| C_{\red}} \in \sM^{n+1}$.

Let us
 consider a subcluster $\hat{S}\subset S$ of colength =1, such that  $\hat{S} \neq S$ precisely at $P$. In particular we ask the ideal $\sI_{\hat{S}}$
to coincide with  $(\sI_{S}, H_{\infty})$, where $H_{\infty} \in {\sI_{(m-1)\Gamma}}_{|P}$.

Define now a 1-dimensional family $\{S_{\lambda}\}$ of clusters, each of them given locally  at $P$ by the ideal
$$\sI_{S_{\lambda}} =   (H+ \lambda H_{\infty}, H_1,\cdots,H_k,  p_1, \ldots, p_l)$$
and coinciding with $\hat{S}$ elsewhere. By construction every $S_{\lambda}$  contains $\hat{S}$
and  we have  $H \not\in H^0(C, \sI_{S_{\lambda}} {K_C})$, which implies
$H^0(C, \sI_{S_{\lambda}} {K_C}) \subsetneq H^0(C, \sI_{\hat{S}} {K_C})$ for every $\lambda \neq 0$.  Indeed, if  locally $H \in {\sI_{S_{\lambda}}}_{|P}$, there would exist elements $\alpha, \alpha_i, \beta_i \in \Oh_{C,P}$ such that
$$H= \alpha (H + \lambda H_{\infty}) + \sum \alpha_i H_i + \sum \beta_i p_i. $$
Since $\{H, H_1, \ldots, H_k\}$ represents a basis for the $\Ka$-vector space $\frac{\sI_{S_{\red}}}{\sI_{S_{\red}}\cap\sM^{n+1}}$, we should  have
$\alpha  \cong \ 1  \ \ mod \mathcal{N}$, the maximal ideal of $\Oh_{C,P}$. In particular $\alpha$ should be  invertible in $\Oh_{C,P}$ and, since $\lambda \in \C^*$, the above equation should imply $$H_{\infty} \in ( H, H_1,\cdots,H_k, p_1,\cdots p_l)= {\sI_{S}}_{|P},$$
i.e., $  {\sI_{\hat{S}}}_{|P} \iso  {\sI_{S}}_{|P}$,  which is impossible by construction of $H_{\infty}$.

On the contrary, since $\cliff{\sI_S K_C}$ is minimal,  it is  $H^0(C, \sI_SK_C) = H^0(C, \sI_{\hat{S}}K_C)$  by our numerical assumptions.
Indeed, let us consider the  residual to    $S$, respectively $\hat{S}$,  with respect to a section in  $H^0(C, \sI_SK_C)$.
 We have $S^{\ast} \subset \hat{S}^{\ast}$ and we know that $S^{\ast}$ satisfies the assumptions of Lemma \ref{bpf} since $S$ does. Hence $h^0(C, \sI_{\hat{S}^{\ast}}K_C) < h^0(C, \sI_{S^{\ast}}K_C)$ and in particular  $h^0(C, \sI_SK_C) = h^0(C, \sI_{\hat{S}}K_C)$ by Riemann-Roch Theorem and Serre duality for residual clusters.

To conclude the proof
we are going to show that this vector space is spanned by
$H$ and a codimension 1 subspace given by sections vanishing on $m\Gamma$.

Our claim is that for every $\lambda\neq 0$  every section in $H^0(C, \sI_{S_{\lambda}} {K_C}) $ vanishes on the curve $m\Gamma$.

Fix a cluster $S_{\lambda}$,  let
$\sigma \in H^0(C, \sI_{S_{\lambda}} {K_C}) $ and consider a generic $S_{\mu}$.
Since both $H^0(C, \sI_{S_{\lambda}} {K_C}) $ and $H^0(C, \sI_{S_{\mu}} {K_C}) $ are codimension 1 subspaces of the same vector space
then there exists a linear combination $\sigma + b_{\mu} H \in H^0(C, \sI_{S_{\mu}} {K_C}) $.

Localizing at $P$, we can write $\sigma= \sum \alpha_i p_i + \alpha (H + \lambda H_{\infty}) +\sum \gamma_i H_i$. Since  $\sigma + b_{\mu} H$
 belongs to $\sI_{S_{\mu}}$ there exists elements $\beta_i,\delta_i$ and $\beta \in \Oh_{C,P}$ such that
$$ \alpha (H + \lambda H_{\infty}) + b_{\mu} H= \sum \beta_i p_i + \beta (H + \mu H_{\infty}) +\sum \delta_i H_i.$$
Both the polynomials are in $\sI_{\hat{S}}$. By the description above, we must have
$$\begin{array}{ll}
\alpha +b_{\mu}= \beta & \operatorname{mod} \sN \\
\alpha\lambda= \beta \mu  &\operatorname{mod} \sN \\
\end{array}$$
where $\mathcal{N}$ as above is the maximal ideal of $\Oh_{C,P}$.
This forces
$$b_{\mu} = \alpha (\operatorname{mod} \mathcal{N}) ( \frac{\lambda}{\mu} - 1).$$

Suppose now that $\alpha \notin \mathcal{N}$.  Then, apart from     $H$,   any element in $\langle \sigma, H\rangle$ should be written as
$a(\sigma+b_{\mu} H) $ for some $\mu$. In particular for $c \neq 0$ every ideal of the form $$ (c \sigma + d H , H_1,\cdots, H_k,   p_1, \ldots, p_l)$$ is contained in some $\sI_{S_{\mu}}$.

This implies that $\length \frac{\Oh_{C,P}}{ (c \sigma + d H , H_1,\ldots, H_k,   p_1, \ldots, p_l)} $ is at least $\length S + 1$ since the ideal vanishes on $S$ and $S_{\mu}$ (since $\sigma \in H^0(C, \sI_{S_{\lambda}} {K_C}) \subset H^0(C, \sI_{S} {K_C})$).

But its degeneration $\frac{\Oh_{C,P}}{  (  H , H_1,\ldots, H_k,   p_1, \ldots, p_l)}= \frac{\Oh_{C,P}}{\sI_{S}} = \Oh_{S} $ has strictly smaller length. This is impossible since the length  is upper semicontinuous.

We must conclude that $\alpha \in \mathcal{N}$ and that $b_{\mu}=0$. This means that the original $\sigma \in H^0(C, \sI_{S_{\lambda}} K_C)$ belongs to $H^0(C, \sI_{S_{\mu}} K_C)$, i.e. $H^0(C, \sI_{S_{\lambda}} K_C)= H^0(C, \sI_{S_{\mu}} K_C)$ for every $\lambda, \mu \in \C^*$.

In particular  every section  in  $H^0(C, \sI_{S_{\lambda}} {K_C})$  must  vanish on every $S_{\mu}$, and in particular
it vanishes on
 the scheme theoretic  union  $\displaystyle{\bigcup_{\mu \in \Ka} S_{\mu}}$ which  has infinite length.
 This may  happen only if  $H^0(C, \sI_{S_{\lambda}} {K_C}))_{| m\Gamma} = \{0\}.$ \end{proof}

\begin{LEM}\label{rank2}
Let $C$ be a   2-connected  projective  curve either reduced with planar singularities or contained in a smooth algebraic surface. Let $B \subset C$ be a subcurve such that the restriction map
$$H^0(C, \sI_{K_{C|B}} K_C) \to H^0(m \Gamma, \Oh_{m \Gamma})$$
has rank 1 for every subcurve $m \Gamma \subset B$.

If $B = \sum_{j=1}^l B_j$ is the decomposition of $B$  in topologically connected component, then the restriction map
$$H^0(C, \sI_{K_{C|B}} K_C) \to H^0(B, \Oh_{B})$$
has rank $\leq l$ (where $l$ is the number of components).
\end{LEM}
\begin{proof} The Lemma follows since the restriction map has rank 1 on every topologically connected component.
\end{proof}

\begin{LEM}\label{rank3}
Let $C$ be a   2-connected  projective  curve either reduced with planar singularities or contained in a smooth algebraic surface. Suppose that $C_{\operatorname{red}}$ is $\mu$-connected. Let $S$ be a subcanonical cluster, and assume that there exists a subcurve $B$ such that $C_{\red} \subset B$ and  the restriction map
$$H^0(C, \sI_{S} K_C) \to H^0(m \Gamma,\sI_{S} K_C)$$
has rank 1 for every subcurve $m \Gamma \subset B$.
Then the following hold.
\begin{enumerate}
\renewcommand\labelenumi{(\roman{enumi})}
\item The restriction map
$H^0(C, \sI_{S} K_C) \to H^0(B, \sI_{S} K_C)$
has rank $k +1 $ (where $k$ is the splitting index of $S$);

\item
If $k>0$ we have
$ \deg K_{C|B} - \deg {S_{|B}} \geq \operatorname{max}\{k;\, \frac{\mu}{2}(k+1)\}.$
\end{enumerate}

\end{LEM}
\begin{proof}
Since the restriction map to every $m \Gamma$ has rank one it is generated by the restriction of a generically invertible section $H \in H^0(C, \sI_{S} K_C)$. By genericity we may assume that $H$ verifies the minimum  for the splitting index, i.e. $H=\sum_{i=0}^k H_i$ with $H_i \in H^0(C,\sI_S K_C)$ and there is a maximal decomposition $C_{\red} = \sum_{i=0}^k C_i$ with $\Supp({H_i}_{|C_{\red}})=C_i$  and $H$ can not be further decomposed.

\hfill\break
({\em i})
To prove the first part of the statement notice
that the restriction map
$$   H^0(C, \sI_{S} K_C)_{|B} \to H^0(C, \sI_{S} K_C)_{|C_{\red}}$$
is an isomorphism.
Indeed the above restriction map is obviously  onto.
 It is injective as well, since otherwise there would be a section $\hat{H}$ in $H^0(C, \sI_{S} K_C)$ vanishing on $C_{\red}$ but not on $B$. i.e.  there would be a subcurve $m \Gamma \subset B$ such that $\hat{H}$ vanishes on $\Gamma$ but not on $m \Gamma$. But the rank of the restriction $H^0(C, \sI_{S} K_C) \to H^0(m \Gamma,\sI_{S} K_C)$ is 1 by our assumptions, as well as the rank of $H^0(C, \sI_{S} K_C) \to H^0(\Gamma,\sI_{S} K_C)$, hence the section $\hat{H}$ can not exists.

Thus without loss of generality we can assume
$B= C_{\red}$ and we take the decomposition $C_{\red} = \sum_{i=0}^k C_i$. \\

The first statement follows if we prove that for every $C_i$ it is $H^0(C, \sI_S K_C)_{|C_i}= \langle H_i \rangle$.
For simplicity we are going to prove it for $C_1$.

Write $C_1 = \sum_{j=1}^{J_0} \Gamma_j$, where $ \Gamma_j$'s are the irreducible components.
Notice that $C_1$ is connected, hence 1-connected, since the decomposition of $C$ is maximal.
We are going to prove by induction that there exists a decomposition sequence
$$ \Gamma_1 = B_1 \subset B_2 \subset \cdots \subset B_{J_0}= C_1$$
such that $H^0(C, \sI_S K_C)_{|B_{J}}= \langle H_1 \rangle$ for every $J\leq J_0$.


The first case, $J=1$, follows from our assumptions.
Assume now  it holds for $B_{J-1}$. Since $C_1$ is 1-connected then $B_{J-1}\cap (C_1 - B_{J-1})\neq \emptyset. $
Take $H_1$ and evaluate it  on $B_{J-1}\cap (C_1 - B_{J-1})$. If it is zero, then $H_1$ can be decomposed as the sum of two sections of $H^0(C, \sI_S K_C)_{|C_1}$, one supported on $B_{J-1}$, the other on $C_1 - B_{J-1}$.  But then we may apply Proposition \ref{generic_splitting}, part 1, to conclude that this would force $H_1$, and $H$ as well, to be decomposed as the sum of more sections than allowed.

Hence there exists at least one component, say $\Gamma_J$, such that $H_1$ does not vanish on $B_{J-1}\cap \Gamma_J$. Define $B_J:=B_{J-1}+\Gamma_J$.
Our claim is that  $H^0(C, \sI_S K_C)_{|B_{J}}= \langle H_1 \rangle$.
If not  there would exist $\overline{H} \in H^0(C, \sI_S K_C)_{|D_J}$  linearly independent from $H_1$ such that   $\overline{H}_{|D_{J-1}} =0$ (possibly after a linear combination of sections). Moreover  we would have $\overline{H}_{|\Gamma_{J}}={H_1}_{|\Gamma_{J}}$ up to rescaling by our assumptions, hence $H_1$ must vanish on $D_{J-1}\cap \Gamma_J$, which is absurd.

\hfill\break
({\em ii})
Suppose now that the splitting index $k$ is at least 1. We are going to study $\deg K_C - \deg S$ on $B$.

Assume at first that $B=C_{\red}$. Consider a decomposition sequence
 $C_0=D_0 \subset D_1 \subset \ldots \subset D_k=B=C_{\red}$ where $D_i - D_{i-1}=C_i$. Up to reindexing the subcurve $C_i$ we can suppose that the curves $D_i$ are topologically connected, hence 1-connected since they are reduced.

We prove by induction that $\deg (\sI_S K_C)_{|D_i} \geq i$. For $i=0$ it is obvious. For $i > 0$ consider the commutative diagram
$$\xymatrix{0 \ar[r] & \sN \ar[r]\ar@{^{(}->}[d] & (\sI_S K_C)_{|D_i} \ar[r]^{\pi_i}\ar@{^{(}->}[d]  & (\sI_S K_C)_{|D_{i-1}} \ar[r]\ar@{^{(}->}[d]   & 0\\
0 \ar[r]& {K_C}_{|C_i}(-D_{i-1}) \ar[r] \ar@{->>}[d] & {K_C}_{|D_i} \ar[r]\ar@{->>}[d]& {K_C}_{|D_{i-1}} \ar[r]\ar@{->>}[d] & 0\\
0 \ar[r]&Z \ar[r] & S_{|D_i} \ar[r] & S_{|D_{i-1}} \ar[r] & 0
}$$
where $\sN$ is the kernel of $\pi_i$ and $Z$ a subsheaf of $S_{|D_i}$, both considered as sheaves with support  on  $C_i$. Notice that by our assumptions  the section $H_i$ restricts to a nonzero generically invertible section of $\sN$, thus $\deg_{C_i} \sN \geq 0$. Computing degrees we obtain
\begin{eqnarray}\label{formula}\deg_{C_i} Z& =& \deg {K_C}_{|C_i}(-D_{i-1}) - \deg \sN \nonumber =\deg {K_C}_{|C_i} - C_i \cdot D_{i-1} - \deg \sN\\
&\leq& \deg {K_C}_{|C_i} - C_i\cdot D_{i-1} \leq \deg {K_C}_{|C_i} -1.\end{eqnarray}

But $\deg S_{|D_i}  = \deg S_{|D_{i-1}} + \deg_{C_i} Z$,
and by induction hypothesis we may assume $\deg (\sI_S K_C)_{|D_{i-1}} \geq (i-1)$. Hence
\begin{eqnarray*}\deg (\sI_S K_C)_{|D_i}& =& \deg {K_C}_{|D_i}- \deg S_{|D_i}\\& =&  (\deg {K_C}_{|D_{i-1}}- \deg S_{|D_{i-1}})+ (\deg {K_C}_{|C_i}-\deg_{C_i} Z)
\geq (i-1)+1 = i
\end{eqnarray*}
In particular we have the first inequality we wanted to prove, i.e. $$ \deg K_{C|B} - \deg {S_{|B}} \geq k.$$

Moreover  Equation (\ref{formula})  yields  $\deg K_{C | C_{i}} - \deg S_{|C_i} \geq  C_i\cdot D_{i-1}$. Taking sum over all  $C_i$'s  we obtain
$$\deg K_{C|C_{\red}} - \deg S_{|C_{\red}} \geq \frac 12 \sum_{i=0}^k C_i\cdot(C_{\red} - C_i)$$
thus if the reduced curve $C_{\red}$ is $\mu$-connected we have
$$ \deg K_{C|B} - \deg S_{|B} \geq \frac{\mu}{2} (k+1).$$
We deal now with the case $C_{\red} \subsetneq B$. We just proved that $$\deg (\sI_S K_C)_{|C_{\red}}=\deg K_{C|C_{\red}} - \deg S_{|C_{\red}}\geq \max \{k, \frac{\mu}{2}(k+1)\}.$$
Consider the following diagram, which exists and commute since $S$ is subcanonical:

$$ \xymatrix{\Oh_{B-C_{\red}}(-(C_{\red})) \ar@{^{(}->}[r] \ar[r]\ar@{^{(}->}[d]& \Oh_B \ar@{^{(}->}[d] \ar@{->>}[r]& \Oh_{C_{\red}}\ar@{^{(}->}[d] \\
\ker(\rho) \ar@{^{(}->}[r]& (\sI_S K_C)_{|B} \ar@{->>}[r]^{\rho} &(\sI_S K_C)_{|C_{\red}}}
$$
Computing degrees we may conclude by the following equation

\begin{eqnarray*}
\deg(\sI_S K_C)_{|B}&=& \chi((\sI_S K_C)_{|B})-\chi(\Oh_B)\\&=&  \chi((\sI_S K_C)_{|C_{\red}})+\chi(\ker(\rho))-\chi(\Oh_{C_{\red}})- \chi(\Oh_{B-C_{\red}}(-(C_{\red})))\\
&=& \deg(\sI_S K_C)_{|C_{\red}} +\chi(\ker(\rho)))- \chi(\Oh_{B-C_{\red}}(-(C_{\red})))\\
&\geq & \max \{k, \frac{\mu}{2}(k+1)\}.
\end{eqnarray*}

\end{proof}

Our main result follows from the following theorem.

\begin{TEO}\label{cliffordfinale}
Let $C$ be a projective  curve  either   reduced with planar singularities  or contained in a smooth  algebraic surface. Assume $C$ to be
 2-connected and $C_{\operatorname{red}}$ $\mu$-connected.

Let $S \subset C$ be a  subcanonical cluster of splitting index $k$.
Then
 \begin{equation}\label{k2} h^0(C,\sI_S K_C) \leq p_a(C) - \frac{1}{2} \deg (S) + \frac{k}{2} \end{equation}

The following holds:
 \begin{enumerate}
  \renewcommand\labelenumi{(\roman{enumi})}
\item if $C_{\operatorname{red}}$ is 2-connected then $h^0(C,\sI_S K_C) \leq p_a(C) - \frac{1}{2} \deg (S) + \operatorname{max}\{0,\, \frac{k}{2} - \frac12\}; $
\item if $C_{\operatorname{red}}$ is 3-connected then $h^0(C,\sI_S K_C) \leq p_a(C) - \frac{1}{2} \deg (S) + \operatorname{max}\{0,\, \frac{k}{4} - \frac34\}; $
\item if $C_{\operatorname{red}}$ is 4-connected then \begin{equation}\label{4-conn}h^0(C,\sI_S K_C) \leq p_a(C) - \frac{1}{2} \deg (S). \end{equation}
 \end{enumerate}

 Moreover if equality holds in Equation (\ref{k2}) or in Equation (\ref{4-conn}) then the pair $(S, C)$ satisfies one of the following assumptions:

 \begin{enumerate}
  \renewcommand\labelenumi{(\roman{enumi})}
\item $S= 0,\, {K_C}$ and $k=0$;
\item  $C$ is honestly hyperelliptic, $S$ is a multiple of the honest $g_{2}^{1}$ and $k=0$;
\item $C$ is 3-disconnected (i.e.  there is
a decomposition $C=A+B$ with  $A\cdot B=2$). \\
\end{enumerate}

\end{TEO}

\begin{proof}
Fix $k \in \Na$ and let $\Sigma_k$ be the set of clusters with splitting index smaller than  or equal to  $k$.

Then Equation (\ref{k2}) corresponds to say that
$$\cliff(\sI_S K_C):= 2p_a(C) -  \deg(S) - 2\cdot  h^0(\sI_S K_C) \geq -k $$
for every cluster $S \in \Sigma_k$.

If the Clifford index of nontrivial clusters is always positive the claim is trivially true. Suppose then the existence of a nontrivial subcanonical cluster with non-positive Clifford index in $\Sigma_k$.\\

\noindent\textbf{Step 1: Clusters of minimal Clifford index and maximal degree.
}
\\

We are going to prove at first that \textit{the claim is true for a cluster of minimal Clifford index and maximal degree}, more precisely that the required inequalities hold for such a cluster and  if equality holds in  Equation (\ref{k2}) or in Equation (\ref{4-conn}) then the pair $(S, C)$ satisfies one of the condition listed in the statement.\\

Let $S$ be a nontrivial subcanonical cluster in $\Sigma_k$ with minimal Clifford index and maximal total degree. Let $S^{\ast}$ be its residual with respect to a generic hyperplane section~$H$. Without loss of generality we can suppose that the splitting index of $S$ is precisely $k$. We have

\begin{equation}\label{M}h^0(C, \sI_S K_C) = p_a(C)-\frac{\deg S}{2} + M\end{equation}
with $M \geq 0$ maximal in $\Sigma_k$.

By Lemma \ref{duality}  we know that either  $S^{\ast}$ is contained in $S$ or $S$ is  disjoint  from $S^{\ast}$ and $\cliff(\sI_S K_C)=0$; in the second case
$S$ and $S^{\ast}$ are Cartier divisors since they are locally isomorphic to $K_C$. \\

\noindent{\em Case 1:  There exists an irreducible component $\Gamma \subset C$ such that}
$$
\dim [ H^0(C, \sI_S K_C)_{|\Gamma}]  \geq 2 \text{\;\;  and   \;\;}
\dim [ H^0(C, \sI_{S^{\ast}} K_C)_{|\Gamma}] \geq 2.$$

By Lemma \ref{saintdonat} we know that $\deg (S^{\ast} )= 2$, that is, $C$ is 3-disconnected or honestly hyperelliptic and that $h^0(C,\sI_S K_C)= p_a(C)-\frac{\deg S}{2}$.  \\

\noindent{\em Case 2:   $S^{\ast} \subset S$ and the restriction map $H^0(C, \sI_S K_C) \to H^0(\Gamma, \sI_S K_C)$ has rank 1 for every  irreducible $\Gamma\subset C$.}
\\

Let $ B = \sum m_i \Gamma_i $ be the minimal subcurve of $C$ containing $S$ and all the $\Gamma_i' s$ such that $\deg_{\Gamma_i} K_C =0$.

First of all notice that $S \cap \Gamma \neq \emptyset$ for every irreducible component $\Gamma \subset C$ such that $K_{C|\Gamma} \neq 0$ because $ S^{\ast}\subseteq S$.  Thus $C_{\red} \subset B$.

By Lemma \ref{rank1}  the restriction map
$H^0(C, \sI_SK_C) \to H^0(m_i\Gamma_i, \sI_SK_C) $ has rank 1 for every irreducible $  \Gamma_i\subset B$ with multiplicity $m_i > 1$ in $B$.
We apply Lemma \ref{rank3} and  we may conclude that
the restriction map
$H^0(C, \sI_SK_C) \to H^0(B, \sI_SK_C) $ has rank $k+1$.\\

Suppose at first that $B \neq C$ (in particular $C$ is not reduced). Consider the following exact sequence
$$ 0 \to \omega_{C-B} \to \sI_S \omega_C \to \sI_S {\omega_{C}}_{|B} \to 0$$
In particular
$$h^0(C,\sI_S K_C)= h^0(C - B, K_{C-B}) + \dim  \im \{r_B: H^0(C,\sI_S K_C) \to H^0(B,\sI_S K_C )\}.$$
Since the restriction map  $r_B$ has rank $k+1$ then

$$h^0(C,\sI_S K_C)= h^0(C-B, K_{C-B}) + k+1.$$
Equation (\ref{genere A+B}) and Equation (\ref{M}) imply that

$$M= k - (\frac{\deg{K_C}_{|B}}{2} - \frac{\deg S}{2}) - (\frac{B \cdot  (C-B)}{2} - h^0(C-B, \Oh_{C-B})).$$

If $k=0$, i.e. the cluster $S$ is not splitting,
every summands in the above formula
can not be positive  since by Lemma \ref{precodimspan} $\frac{B \cdot  (C-B)}{2} - h^0(C-B, \Oh_{C-B}) \geq 0$. Thus we have $M=0$, $S=K_{C|B}$ and  $\frac{B \cdot  (C-B)}{2} = h^0(C-B, \Oh_{C-B})$ and, still by Lemma \ref{precodimspan} we know that the curve $C$ is not 3-connected.

If $k>0$, assume $C_{\red}$ to be $\mu$-connected but not $(\mu +1)$-connected.  By Lemma \ref{rank3} we know that ${\deg{K_C}_{|B}} - {\deg S}\geq \operatorname{max}\{k,\, \frac{\mu}{2}(k+1)\}$, thus  by Lemma \ref{precodimspan}
\begin{eqnarray*}0 \leq M&\leq& \operatorname{\min}\{ \frac{k}{2},\, (1- \frac{\mu}{4})k - \frac{\mu}{4}\} - (\frac{B \cdot  (C-B)}{2} - h^0(C-B, \Oh_{C-B}))\\ & \leq& \operatorname{\min}\{ \frac{k}{2},\, (1- \frac{\mu}{4})k - \frac{\mu}{4}\} . \end{eqnarray*}
 Since $M$ is nonnegative we have that $\mu \leq 3$.

 If $\mu \geq 2$ then $\operatorname{\min}\{ \frac{k}{2},\, (1- \frac{\mu}{4})k - \frac{\mu}{4}\}=(1- \frac{\mu}{4})k - \frac{\mu}{4}$ and
$$\begin{array}{rcl}h^0(C, \sI_S K_C) & \leq& p_a(C)- \frac{\deg S}{2} +  (1- \frac{\mu}{4})k - \frac{\mu}{4}- (\frac{B \cdot  (C-B)}{2} - h^0(C-B, \Oh_{C-B}))\\
& \leq& p_a(C)- \frac{\deg S}{2} +  (1- \frac{\mu}{4})k - \frac{\mu}{4}\end{array}$$
and if equality holds then $C$ is 3-disconnected thanks to Lemma \ref{precodimspan}.
If $C_{\red}$ is 2-disconnected, i.e. $\mu =1$, we know that $\operatorname{\min}\{ \frac{k}{2},\, (1- \frac{\mu}{4})k - \frac{\mu}{4}\}=\frac k2$ and
$$
\begin{array}{rcl}
h^0(C, \sI_S K_C)  & \leq&  p_a(C)- \frac{\deg S}{2} +  \frac k2 - (\frac{B \cdot  (C-B)}{2} - h^0(C-B, \Oh_{C-B})) \\
& \leq& p_a(C)- \frac{\deg S}{2} +  \frac k2 \end{array}$$
and if equality holds then $C$ is 3-disconnected.\\

We have still to study the case in which $B=C$. With the same argument we have
$$M= k - (\frac{\deg{K_C}}{2} - \frac{\deg S}{2}).$$
We can argue as before: if $k=0$ and $M \geq 0$ we have $S= K_C$, which is impossible since we asked $S$ to be nontrivial. If $k>0$ by Lemma \ref{rank3} we conclude that $M \leq \operatorname{min}\{\frac k2, \,(1- \frac{\mu}{4})k - \frac{\mu}{4}\}$ and $C_{\red}$ is not 4-connected. Moreover if $M=\frac k2$ then $\deg K_C - \deg S=k$. But   $\deg S^{\ast}=\deg K_C - \deg S=k$ by its definition. This forces $h^0(C, \sI_{S^{\ast}} K_C) = p_a(C)-\frac{\deg S^{\ast}}{2} + \frac k2=p_a(C)$ which is impossible since $k > 0$ and $K_C$ ample. Thus $M \leq \operatorname{min}\{\frac k2 - \frac 12, \,(1- \frac{\mu}{4})k - \frac{\mu}{4}\}$ and
$$h^0(C, \sI_S K_C) \leq p_a(C)- \frac{\deg S}{2} +  (1- \frac{\mu}{4})k - \frac{\mu}{4}$$
if $\mu=2,\,3$ while
$$h^0(C, \sI_S K_C) \leq p_a(C)- \frac{\deg S}{2} +  \frac k2 - \frac 12$$
if $\mu=1$.\\

\noindent{\em Case 3: $S$ is a Cartier divisor, $\cliff(\sI_S K_C)=0$ and there exists a decomposition $C=A+B$ such that
$A$ and $B$ have no common components, $S= {K_C}_{|B}$ and $ S^{\ast}= {K_C}_{|A}$. }
\\
If Case 1 and 2 do not hold  we may conclude by Lemma \ref{duality} that $S$ and $S^{\ast}$ are disjoint Cartier divisor, that their Clifford index is zero, and that for every irreducible $\Gamma \subset C$ one of the restriction maps to $H^0(\Gamma, \sI_S K_C)$ and $H^0(\Gamma, \sI_{S^{\ast} }K_C)$ has rank one.

If for  an irreducible $\Gamma\subset C$  the restriction map
$H^0(C, \sI_S K_C) \to H^0(\Gamma, \sI_S K_C)$ has rank one, since $\sI_S K_C$ is base point free by Proposition \ref{bpf}, then  $S_{|\Gamma} = {K_C}_{|\Gamma}$ and moreover $S_{| n\Gamma} = {K_C}_{|n\Gamma}$
for $\Gamma $ of multiplicity  $n$ since $S$ is Cartier. Thus $S^{\ast}_{|n\Gamma}= \emptyset$.

The same holds for $S^{\ast}$.
Therefore there exists a decomposition $C=A+B$ such that
$A$ and $B$ have no common components, $S= {K_C}_{|B}$ and $ S^{\ast}= {K_C}_{|A}$.\\

Notice that in this case, since $S$ and $S^{\ast}$ are Cartier divisor with minimal Clifford index, by Proposition \ref{bpf} we know that $|\sI_S K_C|$ and  $|\sI_{S^{\ast}} K_C|$ are base point free and in particular the generic section does not pass through the singularities of $C_{\red}$. Thus the splitting index $k$ is 0.

In this situation we consider the following exact sequences
$$0 \to \omega_A \to \sI_S \omega_C \stackrel{r_B}{\longrightarrow} \Oh_B \to 0$$
$$0 \to \omega_B \to \sI_{S^{\ast}} \omega_C \stackrel{r_A}{\longrightarrow} \Oh_A \to 0.$$

Since $h^0(C, \sI_S K_C)= h^0(A, K_A)+ \rank(r_B)$ (and similarly for $S^{\ast}$) the conditions $\cliff(\sI_S K_C)=\cliff(\sI_{S^{\ast}} K_C)=0$ imply that
\begin{eqnarray*} h^0(A, \Oh_A)+ \rank(r_B) = \frac{A\cdot  B}{2}+1 \\
h^0(B, \Oh_B)+\rank(r_A)= \frac{A\cdot  B}{2}+1 \nonumber
\end{eqnarray*}
hence
\begin{equation}\label{ranghi}
h^0(A, \Oh_A) + h^0(B, \Oh_B)+  \rank(r_A)+\rank(r_B) = {A\cdot  B}+2
\end{equation}

Write $A= \sum_{i=1}^h A_i$ and $B=\sum_{j=1}^l B_j$ where the $A_i$ and $B_j$ are the topologically connected components of $A$, respectively $B$.

By Lemma \ref{rank1} and \ref{rank2} we know that $\rank(r_A) \leq h$ and $\rank(r_B) \leq l$.\\

If $h^0(A, \Oh_A) \leq \frac{A\cdot  B}{2}- h$ and $h^0(B, \Oh_B) \leq \frac{A\cdot  B}{2}- l$
equation (\ref{ranghi})  implies that
$${A\cdot  B}+2  \leq \frac{A\cdot  B}{2}- h + \frac{A\cdot  B}{2}- l + h+l$$
which is impossible. \\

Thus we have  either $h^0(A, \Oh_A) > \frac{A\cdot  B}{2}- h$ or $h^0(B, \Oh_B) > \frac{A\cdot  B}{2}- l$. Let us suppose that the first inequality is true.

By Lemma \ref{precodimspan} we have $h^0(A, \Oh_A) \leq \frac{A\cdot  B}{2}- \frac{m-2}{2}\cdot h$ assuming C $m$-connected (with $m \geq 3$). Therefore we know that $C$ is 4-disconnected and moreover there must be a topologically connected component of $A$, say $A_1$, such that $h^0(A_1, \Oh_{A_1}) > \frac{A_1\cdot  B}{2}- 1$.\\

If $C$ is 3-connected, Lemma \ref{precodimspan} says that $h^0(A_1, \Oh_{A_1}) \leq \frac{A_1\cdot  B}{2}- \frac12$ and we must conclude that
\mbox{$h^0(A_1, \Oh_{A_1})=1$} and $A_1\cdot  B=A_1\cdot (C-A_1)=3$. This forces $C-A_1$ to be 2-connected by \cite[Lemma A.4]{CFM} and allows us to consider the subcanonical cluster  $\tilde{S}^{\ast}:=S^{\ast}\cap (C-A_1)= K_{(C-A_1)|(A-A_1)}$. It is
 $$h^0(C, \sI_{S^{\ast}}K_C)= h^0(C-A_1, \sI_{\tilde{S}^{\ast}}K_{C-A_1})+1. $$
By an induction argument, we apply Clifford's theorem to the curve $C-A_1$ and the cluster $\tilde{S}^{\ast}$ which can be easily seen to be subcanonical for
the system $|K_{(C-A_1)}|$ since $\Oh_{C-A_1} \subset {\sI}_{\tilde{S}^{\ast}} K_{C-A_1}$. Moreover the splitting index of ${\tilde{S}^{\ast}}$ is zero since it is clear that $H^0(C-A_1, \sI_{\tilde{S}^{\ast}}K_{C-A_1})$ does not have any base point in $\Sing((C-A_1)_{\red})$. Thus we have
$$h^0(C, \sI_{S^{\ast}}K_C)= h^0(C-A_1, \sI_{\tilde{S}^{\ast}}K_{C-A_1})+1 \leq p_a(C-A_1) - \frac{\deg(\tilde{S}^{\ast})}{2}+1.$$
Since $p_a(C-A_1) = p_a(C) - p_a(A_1) -2$ and \mbox{$\deg(\tilde{S}^{\ast})=\deg({S}^{\ast})- (2p_a(A_1) +1)$,}  we conclude that
\begin{eqnarray*}h^0(C, \sI_{S^{\ast}}K_C)&\leq& (p_a(C) - p_a(A_1) -2) - \frac{\deg({S}^{\ast})}{2} +(p_a(A_1) +\frac12) +1\\&=&p_a(C) - \frac{\deg({S}^{\ast})}{2} - \frac12.
\end{eqnarray*}
Therefore $M= - \frac 12$, but we were asking $M \geq 0$, hence $C$ is 3-disconnected.\\

\noindent\textbf{Step 2: Clusters of minimal Clifford index of any  degree.
}
\\

We deal now with the case of a \textit{cluster $S$ of minimal Clifford index, without any assumption on its degree.}

If there exists a nontrivial cluster with minimal nonpositive Clifford index $S \in \Sigma_k$, there exists as well a nontrivial cluster $S_{\operatorname{max}}$ of maximal degree with the same Clifford index. In particular, a straightforward computation shows that the inequalities of the statement hold for $\sI_S K_C$ if and only if they hold for $\sI_{S_{\operatorname{max}}} K_C$, and similarly for the equalities.

We just showed that $\sI_{S_{\operatorname{max}}} K_C$, and thus $\sI_S K_C$ as well, satisfies the inequalities of the statement, hence proving the first part of the statement.\\

Moreover, if equality holds in  Equation (\ref{k2}) or in Equation (\ref{4-conn}) for $\sI_S K_C$ (and, equivalently, for $\sI_{S_{\operatorname{max}}} K_C$), then the pair $(S_{\operatorname{max}}, C)$ satisfies one of the condition listed in the statement. If $C$ is 3-disconnected there is nothing more to prove.

If, instead, $C$ is 3-connected, then case \textit{(ii)} must hold, hence $C$ is honestly hyperelliptic. We can repeat verbatim the classical idea of Clifford's theorem for a smooth hyperelliptic curve of Saint Donat (see \cite{sd} or \cite[Lemma IV.5.5]{Ha}) and conclude that $S$ is a multiple of a honest $g_2^1$.\end{proof}\\

As a corollary we obtain the following result in which the computation of the splitting index, usually tricky, is avoided by the count of the number of irreducible components.

\begin{TEO}\label{numero_irriducibili}
Let $C= \sum_{i=0}^s n_i \Gamma_i$ be a projective curve  either   reduced with planar singularities  or contained in a smooth  algebraic surface with $(s+1)$ irreducible components. Assume $C$ to be
 2-connected
 and let $S\subset C$ be a  subcanonical cluster. Then
 $$h^0(\sI_S K_C) \leq p_a(C) - \frac{1}{2} \deg (S) + \frac s2.$$
\end{TEO}
\begin{proof}
If follows immediately from Theorem \ref{cliffordfinale} since the splitting index of every cluster is at most the number of irreducible components of $C$ minus 1.
\end{proof}\\

If $S$ is a Cartier divisor we have the following theorem.

\begin{TEO}\label{Clifford_Cartier} Let $C$ be a projective  curve either reduced with planar singularities or contained in a smooth algebraic surface. Assume $C$ to be 2-connected and let $S\subset C$ be a subcanonical Cartier cluster. Then

$$h^0(C, \sI_S K_C) \leq p_a(C)-\frac 12 \deg(S).$$

Moreover if equality holds then the pair $(S,C)$ satisfies one of the following assumptions:
\begin{enumerate}
  \renewcommand\labelenumi{(\roman{enumi})}
\item $S= 0,\, {K_C}$;
\item  $C$ is honestly hyperelliptic and $S$ is a multiple of the honest $g_{2}^{1}$;
\item $C$ is 3-disconnected (i.e.  there is
a decomposition $C=A+B$ with  $A\cdot B=2$). \\
\end{enumerate}
\end{TEO}

\begin{proof}
If $S$ is not splitting the results follows from Theorem \ref{cliffordfinale}. Thus we can suppose that $S$ has splitting index $k > 0$. By Proposition \ref{generic_splitting} we know that there is a decomposition $C_{\red}= \sum_{i=0}^k C_i$ such that every $H \in H^0(C, \sI_S K_C)$ can be written as $H=\sum_{i=0}^k H_i$ with $H_i \in H^0(C, \sI_S K_C)$ and $\Supp H_i \subset C_i$.

In particular every section $H \in H^0(C,\sI_S K_C)$ vanishes on $C_i \cap C_j$ and we can decompose  $H^0(C, \sI_S K_C)_{|C_{\red}}$ as the direct sum of proper subspaces.

$$H^0(C, \sI_S K_C)_{|C_{\red}} = \bigoplus_{i=0}^k H^0(C,\sI_S K_C)_{|C_i}$$
such that the following diagram holds:

$$\xymatrix{\bigoplus_{i=0}^k H^0(C,\sI_S K_C)_{|C_i} \ar@{^{(}->}[d]  \ar[r]^{\iso}& H^0(C,\sI_S K_C)_{|C_{\red}} \ar@{^{(}->}[d] &\\
\bigoplus_{i=0}^k H^0(C_i, \sI_S \cap \sI_{C_i \cap (C_{\red}-C_i)} K_{C|C_i})\ar[r]& H^0(C_{\red},\sI_S K_{C|C_{\red}})\ar[r]& H^0(Z, \Oh_Z)} $$

Since the map $\bigoplus_{i=0}^k \sI_S \cap \sI_{C_i \cap (C_{\red}-C_i)} \omega_{C|C_i} \to \sI_S \omega_{C|C_{\red}}$ is generically an isomorphism its cokernel is a skyscraper sheaf $\Oh_Z$. Since $S$ is Cartier, it is not difficult to verify that $\Oh_Z$ is isomorphic, as sheaf on $C_{\red}$, to the structure sheaf of the scheme $\bigcup_{i,j} C_i \cap C_j$, thus it has length $\frac 12 \sum_{i=0}^k C_i\cdot(C_{\red} - C_i)$.\\

Let $\overline{S}$ be the base locus of $H^0(C,\sI_S K_C)$. We have the following exact sequence
$$ 0 \to \sI_{\overline{S}} \to \sI_S \to \sF \to 0$$
and $\sF \cong \Oh_{\xi}$ where $\xi$ is a cluster. It is clear from the above diagram that there is a natural surjective morphism  $ \Oh_{\xi} \onto \Oh_Z$. In particular the colength of  $S\subset \ \overline{S}$  is at least $\frac 12 \sum_{i=0}^k C_i\cdot(C_{\red} - C_i) \geq k$.

Since $H^0(C, \sI_S K_C)= H^0(C, \sI_{\overline{S}} K_C)$ the splitting index of $\hat{S}$ is still $k$ and we can apply Theorem \ref{cliffordfinale}:

$$\begin{matrix}h^0(C, \sI_S K_C)&=& h^0(C, \sI_{\overline{S}} K_C) \leq p_a(C)- \frac 12 \deg{\overline{S}} + \frac k2 \\
&=&p_a(C)- \frac 12 \deg{S} - \frac 12 \operatorname{colength} (\overline{S} \supset S)+ \frac k2 \\
& \leq & p_a(C)- \frac 12 \deg{S} - \frac k2 + \frac k2= p_a(C)- \frac 12 \deg{S}.\end{matrix}$$

Notice that if equality holds $h^0(C, \sI_{\overline{S}} K_C) = p_a(C)- \frac 12 \deg{\overline{S}} + \frac k2$, thus by Theorem \ref{cliffordfinale} we know that one of the 3 cases listed ($\overline{S}$ trivial, or $C$ honestly hyperelliptic, or $C$ 3-disconnected) must hold. Since we are assuming that the splitting index $k$ is strictly positive, we are forced to conclude that case $\textit{(iii)}$ of Theorem \ref{cliffordfinale} holds, i.e., $C$ is 3-disconnected.
\end{proof}

\hfill\break{\bf Proof of Theorem A.}
It is a straightforward corollary of Theorem \ref{cliffordfinale} if  $C_{\red}$ is 4-connected; of Theorem \ref{Clifford_Cartier}  if $S$ is Cartier; of Proposition \ref{generic_splitting} and Theorem \ref{cliffordfinale} if  there is a section $H \in H^0(C, \sI_S K_C)$ avoiding the singularities of $C_{\red}$ since  in this case $S$ is not splitting.
{\hfill {\bf Q.E.D.}}

\section{Clifford's theorem for reduced curves}

In this section we will prove Clifford's theorem for reduced 4-connected  curves with planar singularities. Theorem B works under the assumptions that the sheaves $\sI_S L$ and its dual $\sHom(\sI_S L, \omega_C)$ are NEF.

In Theorem C we deal with the case in which the second sheaf is not NEF. We split the curve in $C_0 + C_1$ where $C_1$ is the NEF part. It is still possible to find a Clifford type bound for $h^0(C, \sI_S L)$ with a correction term which corresponds to a Riemann-Roch estimate over $C_0$. In the extremal case in which $C=C_0$ we recover Riemann-Roch Theorem since $h^1(C, \sI_S L)=0.$

The inequality of Theorem C can be written also as

$$h^0(C, \sI_S L) \leq \frac{\deg(I_S L)_{|C_1}}{2} + \deg(I_S L)_{|C_0} - \frac{\deg(K_C)_{|C_0}}{2}.$$

The following trivial remark will be useful in the proof of Theorem B and C.
\begin{REM}\label{ovvio} Let $C$ be a reduced projective curve with planar singularities. Let $C= A+B$ be an effective decomposition of $C$ in non trivial subcurves. Consider two rank one torsion free sheaves $\sI_{S_A} L_A$ and $\sI_{S_B} L_B$ supported respectively on $A$ and $B$ with the property that $A \cap B \subset S_A$ and $A \cap B \subset S_B$. Then  the sheaf on $C$ defined as $\sI_{S_A} L_A \oplus \sI_{S_B} L_B$ is a rank one torsion free sheaf as well, since the sheaves living on the two curves can be glued together as they both vanish on the intersection.

\end{REM}

\hfill\break{\bf Proof of Theorem B.}
If $H^0(C,\sI_S L )=0$ or $H^1(C,\sI_S L )=0$ the result follows from Riemann-Roch Theorem and the positivity of $\deg{\sI_S L}$. We will assume from now on that both spaces are nontrivial.\\

We are going to show that if the sheaf $\sI_S L$ attains the minimal Clifford index among the sheaves satisfying the assumption  of Theorem B, then $\sI_S L$ is a subcanonical sheaf.

With this aim we prove firstly  that there exists an inclusion $\Oh_C \into \sI_S L$ and secondly that there exists an inclusion $ \sI_S L \into \omega_C$ .\\

If  $\Oh_C\not\into \sI_S L$,
let $B \subset C$ be the maximal subcurve  which annihilates every section in $H^0(C,\sI_S L)$ 
and let $A=C-B$. Then  there is a cluster $S_A$ on $A$ such that
$$0 \to \sI_{S_A} L_{|A}(-B) \to \sI_S L \to (\sI_S L)_{|B}\to 0$$
and moreover there is an isomorphism between vector spaces:
$$H^0(A, \sI_{S_A} L_{|A}(-B)) \cong H^0(C,\sI_S L ).$$

If $A \neq C$, consider the sheaf $\sF= \sI_{S_A} L_{|A}(-B) \oplus \Oh_B(A)(-A)$. By Remark \ref{ovvio} $\sF$ is a rank 1 torsion free sheaf and it is immediately seen that
$$0 \leq \deg sF_{|C_0} \leq K_{C|C_0} \text{  for every subcurve } C_0 \subset C.$$

Since $\cliff(\sI_S L)$ is minimum by our assumption then
\begin{equation}\label{disuguaglianza}\cliff (\sI_S L) \leq \cliff(\sF).\end{equation}

But, by our construction $h^0(C, \sF)= h^0(A, \sI_{S_A} L_{|A}(-B))+h^0(B, \Oh_B)$ and by definition of degree we have
\begin{eqnarray*}\deg(\sI_S L)& = &\chi (\sI_S L)- \chi(\Oh_C)= \chi (\sI_{S_A} L_{|A}(-B))+ \chi ((\sI_S L)_{|B}) -\chi (\Oh_C)\\
&=& \chi (\sI_{S_A} L_{|A}(-B))+ \deg ((\sI_S L)_{|B})+\chi(\Oh_B) -\chi (\Oh_C)\\
&\geq &\chi (\sI_{S_A} L_{|A}(-B))  +\chi(\Oh_B) -\chi (\Oh_C) = \deg (\sF).
\end{eqnarray*}
Thus
\begin{eqnarray*}\cliff(\sF) &=& \deg(\sF)- 2h^0(C, \sF)+2 \\&\leq &\deg(\sI_S L) - 2h^0(A, \sI_{S_A} L_{|A}(-B))-2h^0(B, \Oh_B)+2\\
&\leq &\deg(\sI_S L) - 2h^0(A, \sI_{S_A} L_{|A}(-B)) = \cliff(\sI_S L)-2.
\end{eqnarray*}
This contradicts Equation  (\ref{disuguaglianza}), hence $A=C$, i.e., there exist sections not vanishing on any subcurve, or, equivalently, $\Oh_C \into \sI_S L$.\\

Now  we show  that $\sI_S L \into \omega_C$. The dual sheaf $\sHom (\sI_S L, \omega_C)$ satisfies the assumption of Theorem B and by Serre duality  it has the same Clifford index of $\sI_S L$, hence thanks to the previous step $\Oh_C \into \sHom (\sI_S L, \omega_C)$. In particular $H^0(C, \Oh_C) \into H^0(C, \sHom (\sI_S L, \omega_C))= \Hom (\sI_S L, \omega_C)$. Hence there is a map from $\sI_S L$ to $\omega_C$ not vanishing on any component, and by automatic adjunction (Proposition \ref{lem:adj}) we may conclude that $\sI_S L \into \omega_C$.\\

We proved that any sheaf $\sI_S L$ with minimal Clifford index satisfies $\Oh_C \into \sI_S L \into \omega_C$, hence $\sI_S L \iso \sI_T \omega_C$ where $T$ is a subcanonical cluster. But Theorem A holds for $\sI_T \omega_C$, which concludes the proof.{\hfill {\bf Q.E.D.}}\\

\begin{REM}\label{h0-nonsubcanonico}
If  $ \sI_S L$ is not isomorphic to a sheaf of the form $ \sI_T K_C$ for a subcanonical $T$ then the proof of Theorem B shows that we have the stricter inequality  $h^0(C, \sI_S L)  \leq  \frac{\deg (\sI_S L)}{2}$.

\end{REM}
\hfill\break {\bf Proof of Theorem C.}
If $H^0(C,\sI_S L )=0$ or $H^1(C,\sI_S L )=0$ the result follows from Riemann-Roch Theorem and the positivity of $\deg{\sI_S L}$. We will assume from now on that both spaces are nontrivial.

Let $C_0$ be the maximal subcurve such that $$\deg[(\sI_S L)]_{|B} > \deg {K_C}_{|B}$$
for every subcurve $B \subset C_0$.

Consider a cluster $T$ on $C_0$ such that $(\sI_T \sI_S L)_{|C_0} \numeq K_{C|C_0}$. Such cluster must exist by our degree assumptions.
The sheaf $\sI_T \sI_S L$ satisfies the assumptions of Theorem B, thus $$h^0(C, \sI_T \sI_S L ) \leq  \frac{\deg \sI_S L }{2} - \frac{\deg T}{2} +1.$$

Moreover, if $\sI_T \sI_S L \iso \sI_Z \omega_C$ with $Z$ subcanonical cluster, we have $$h^1(C, \sI_T \sI_S L ) > h^1(C, \sI_S L ).$$ This follows from the analysis of the following commutative diagram:

$$\xymatrix{ H^1(C,\sI_S L )^* \iso\Hom_C (\sI_S L, \omega_C)  \ar@{^{(}->}[r] \ar[d]& \Hom_C (\sI_T \sI_S L, \omega_C)\ar[d]^{r_0} \iso H^1(C,\sI_T \sI_S L )^*\\
0=\Hom_{C_0}(\sI_S L,  \omega_C)  \ar[r] & \Hom_{C_0}(\sI_T\sI_S L,  \omega_C)  }$$

We have that $\Hom_{C_0}(\sI_S L,  \omega_C)= H^1(C_0, \sI_S L)^*=0$ by Corollary \ref{h1=0}. The map $r_0$ corresponds to the restriction map $H^0(C, \sI_{Z^{\ast}}K_C) \to H^0(C_0, \sI_{Z^{\ast}}K_C)$, which  is nonzero since $Z^{\ast}$ is subcanonical.

In this case we may conclude since
$$h^0(C, \sI_S L ) \leq h^0(C, \sI_T \sI_S L ) + \deg T -1 \leq \frac{\deg \sI_S L }{2}  + \frac{\deg (\sI_S L -  K_C)_{|C_0}}{2}.$$

If $\sI_T \sI_S L$ is not subcanonical, by Remark \ref{h0-nonsubcanonico} it is $h^0(C, \sI_T \sI_S L ) \leq \frac{\deg \sI_S L }{2} - \frac{\deg T}{2} $ and we have the same inequality.

{\hfill {\bf Q.E.D.}}

\section{Examples}

In this section we will illustrate some examples in which  the  estimates of Theorem \ref{cliffordfinale} and Theroem B and C are sharp. The first two examples concern Theorem \ref{cliffordfinale} and show that the Clifford index can be negative when $C_{\red}$ is 4-disconnected. Examples 5.3 and 5.4 regard Theorem A and in particular they show how to build 3-disconnected Gorenstein curves and nontrivial and non splitting subcanonical clusters with vanishing Clifford index.
The final example (firstly given by L. Caporaso in \cite[Ex. 4.3.4]{CAP}) shows a case in which \mbox{$H^0(C, \sI_S L) \neq 0$,} $H^1(C, \sI_S L) \neq 0$ and equality holds in Theorem C.

\begin{EX} Let $C= \sum_{i=0}^k \Gamma_i$ such that $\Gamma_i \cdot \Gamma_{i+1}=1$, $\Gamma_0 \cdot \Gamma_k=1$ and all the other intersection products are 0. Suppose that $p_a(\Gamma_i) \geq 2$.

$$
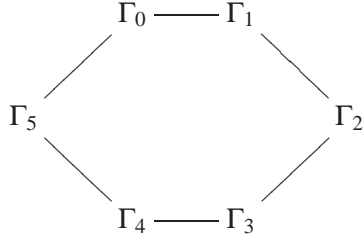
\xymatrix{& \Gamma_0 \ar@{-}[r] \ar@{-}[dl] & \Gamma_1\ar@{-}[dr] &\\
\Gamma_5 \ar@{-}[dr] & & &\Gamma_2 \ar@{-}[dl]\\
& \Gamma_4 \ar@{-}[r]& \Gamma_3 &}$$
Take $S^{\ast}= \bigcup_{i,j} (\Gamma_i \cap \Gamma_j)$, which is a degree $(k+1)$ cluster. Since $\Gamma_i \cdot (C- \Gamma_i)=2$ every section in $H^0(C, K_C)$ vanishing on a singular point of $\Gamma_i$ must vanish on the other, hence if a section $H^0(C, K_C)$ vanishes on any of such points must vanish through all of them. In particular $h^0(C, \sI_{S^{\ast}} K_C)= p_a(C)-1 = p_a(C)- \frac{\deg S^{\ast}}{2} + \frac{k}{2}- \frac12$. It is clear that the splitting index of $S^{\ast}$ is precisely $k$.
\end{EX}

\begin{EX}\label{tetraedro} Let $C= \sum_{i=0}^5 \Gamma_i$ and suppose that $p_a(\Gamma_i) \geq 2$. Suppose moreover that the intersection products are defined by the following dual graph, where the existence of the simple line means that the intersection product between the two curves is 1.
 $$
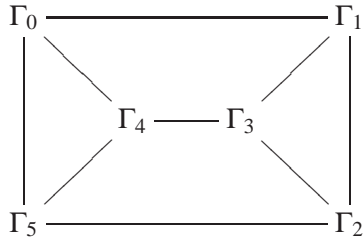
\xymatrix{\Gamma_0  \ar@{-}[rrr]  \ar@{-}[dr]  \ar@{-}[dd]&&& \Gamma_1  \ar@{-}[dl]  \ar@{-}[dd]\\
 & \Gamma_4  \ar@{-}[r]  \ar@{-}[dl]& \Gamma_3  \ar@{-}[dr] &\\
 \Gamma_5  \ar@{-}[rrr]&&& \Gamma_2
}$$
Take $S^{\ast}= \bigcup_{i,j} (\Gamma_i \cap \Gamma_j)$, which is a degree 9 cluster. It is easy to check that one can decompose $H^0(C, \sI_{S^{\ast}}K_C) \iso \oplus_{i=0}^5 H^0(\Gamma_i, K_{\Gamma_i})$ and that the splitting index of $S^{\ast}$ is $k=5$. Thus we have
$$h^0(C, \sI_{S^{\ast}}K_C)= \sum_{i=0}^s p_a(\Gamma_i) = p_a(C)- \frac{\deg S^{\ast}}{2} + \frac 12$$
and notice that $\frac12$ is precisely $\frac{k}{4} -\frac34$, which means that equality can hold when $C_{\red}$ is 3-connected but 4-disconnected.
\end{EX}

\begin{EX} Let $C=  \Gamma_0 + \sum_{i=1}^n\Gamma_i$ with $\Gamma_0 \cdot  \Gamma_i=2$ for every $i \geq 1$ and $\Gamma_i \cdot \Gamma_j=0$ for $i>j \geq 1$ (possibly $\Gamma_i = \Gamma_j$ for some $i,\,j$).

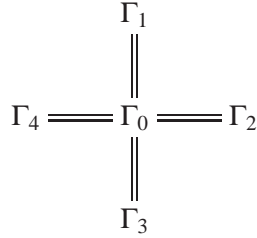
$$\xymatrix{&\Gamma_1 \ar@2{-}[d]&\\
\Gamma_4 \ar@2{-}[r]&  \Gamma_0 \ar@2{-}[d]& \ar@2{-}[l] \Gamma_2\\
& \Gamma_3&}$$

$C$ is 2-connected but 3-disconnected. Taking $S= {K_C}_{|C- \Gamma_0}$ we have
$$h^0(C, \sI_S K_C)= h^0(\Gamma_0, K_{\Gamma_0}) + h^0(C- \Gamma_0, \Oh_{C- \Gamma_0})=p_a(C)- \frac{\deg S}{2}$$
since $h^0(C- \Gamma_0, \Oh_{C- \Gamma_0})=n=\frac{\Gamma_0 \cdot (C-\Gamma_0)}{2}$.
\end{EX}

\begin{EX} Let $C= \Gamma_0 + \Gamma_1$ with $\Gamma_1$ irreducible and $\Gamma_0$ irreducible and hyperelliptic. Suppose that $\Oh_{\Gamma_0} (\Gamma_1)$ is a $g_2^1$ divisor on $\Gamma_0$.

 Let $S$ be another divisor in the linear series $g_2^1$ on $\Gamma_0$. Then $h^0(\Gamma_0, \sI_S K_C)= p_a(\Gamma_0)$ thus $h^0(C, \sI_S K_C) = p_a(\Gamma_1) + p_a(\Gamma_0) = p_a(C)-1= p_a(C)- \frac{\deg S}{2}$.
\end{EX}

We believe that if $C$ is 2-connected but 3-disconnected and $\cliff(\sI_S K_C)=0$ for a subcanonical non splitting cluster $S$,  then $S$ must be the sum of clusters shaped as the two above, i.e. a linear combination of a sum of $g_2^1$ plus a term of the form $K_{C|B}$ with $h^0(B, \Oh_B)= \frac{B \cdot  (C-B)}{2}$ or $h^0(C-B, \Oh_{C-B})= \frac{B \cdot  (C-B)}{2}$.\\

\begin{EX} Let $C= \sum_{i=1}^k \Gamma_i + \sum_{j=1}^k E_j$ where $p_a(\Gamma_i)=0$, $p_a(E_j)=1$. Moreover $\Gamma_i \cdot E_i = \Gamma_i \cdot  E_{i-1}= \Gamma_1\cdot E_k=1$ and every other intersection number is 0.

$$
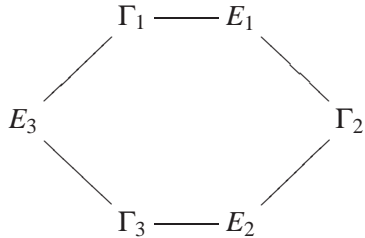
\xymatrix{& \Gamma_1 \ar@{-}[r] \ar@{-}[dl] & E_1\ar@{-}[dr] &\\
E_3 \ar@{-}[dr] & & &\Gamma_2 \ar@{-}[dl]\\
& \Gamma_3 \ar@{-}[r]& E_2 &}$$

Take a smooth point $P_i$ over each $\Gamma_i$ and consider the sheaf $L = \Oh_C (\sum_i P_i)$. Under the notation of Theorem $C$, we have that $C_0= \sum \Gamma_i$ and $C_1= \sum E_j$. A straightforward computation shows that $h^0(C, L)= k= \frac{\deg L}{2} + \frac{\deg (L-K_C)_{|C_0}}{2}$.
\end{EX}

  \vspace{.5cm}
Marco Franciosi\\
Dipartimento di Matematica, Universit\`a di Pisa\\
Largo B.Pontecorvo 5,  I-56127 Pisa (Italy)\\
{\tt franciosi@dm.unipi.it}
\vspace{.5cm}\\
Elisa Tenni\\
SISSA International School for Advanced Studies\\
via Bonomea 265, I-34136 Trieste (Italy)\\
{\tt etenni@sissa.it}

\end{document}